\documentclass{article}

\usepackage{amscd,amsmath,amssymb,amsthm,xcolor,mathrsfs,dsfont}
\newcommand{\eps}{\varepsilon}

\newcommand{\R}{\mathbb R}
\newcommand{\N}{\mathbb N}

\newcommand{\then}{\Longrightarrow}
\newcommand{\J}{{\cal J}}

\newcommand{\M}{{\cal M}}
\newcommand{\T}{{\cal T}}
\newcommand{\er}{{\cal R}}
\newcommand{\enne}{{\cal N}}

\newcommand\meas{{\rm meas}}

\DeclareMathOperator{\codim}{codim}
\DeclareMathOperator*{\esssup}{ess\; sup}

\newtheorem{corollary}{Corollary}[section]
\newtheorem{theorem}[corollary]{Theorem}
\newtheorem{lemma}[corollary]{Lemma}
\newtheorem{proposition}[corollary]{Proposition}

\theoremstyle{definition}
\newtheorem{definition}[corollary]{Definition}
\newtheorem{remark}[corollary]{Remark}
\newtheorem{example}[corollary]{Example}

\numberwithin{equation}{section}

\begin{document}

\title{{\bf Existence and multiplicity results\\
for a class of coupled quasilinear elliptic systems\\
of gradient type}
\footnote{The research that led to the present paper was partially supported 
by MIUR--PRIN project ``Qualitative and quantitative aspects of nonlinear PDEs'' (2017JPCAPN\underline{\ }005),
{\sl Fondi di Ricerca di Ateneo} 2017/18 ``Problemi differenziali non lineari''  }}

\author{A.M. Candela, A. Salvatore, C. Sportelli\\
{\small Dipartimento di Matematica}\\
{\small Universit\`a degli Studi di Bari Aldo Moro} \\
{\small Via E. Orabona 4, 70125 Bari, Italy}\\
{\small \it annamaria.candela@uniba.it}\\
{\small \it addolorata.salvatore@uniba.it}\\
{\small \it caterina.sportelli@uniba.it}}
\date{}

\maketitle

\begin{abstract}
The aim of this paper is investigating the existence of one or more 
weak solutions of the coupled quasilinear elliptic system of gradient type
\[
(P)\qquad \left\{
\begin{array}{ll}
- {\rm div} (A(x, u)\vert\nabla u\vert^{p_1 -2} \nabla u) 
+ \frac{1}{p_1}A_u (x, u)\vert\nabla u\vert^{p_1} = G_u(x, u, v) &\hbox{ in $\Omega$,}\\ [10pt]
- {\rm div} (B(x, v)\vert\nabla v\vert^{p_2 -2} \nabla v) 
+\frac{1}{p_2}B_v(x, v)\vert\nabla v\vert^{p_2} = G_v\left(x, u, v\right) &\hbox{ in $\Omega$,}\\[10pt]
u = v = 0 &\hbox{ on $\partial\Omega$,}
\end{array}
\right.
\]
where $\Omega \subset \R^N$ is an open bounded domain, $p_1$, $p_2 > 1$ and
$A(x,u)$, $B(x,v)$ are $\mathcal{C}^1$--Carath\'eo\-do\-ry functions on $\Omega \times \R$
with partial derivatives $A_u(x,u)$, respectively $B_v(x,v)$, while
$G_u(x,u,v)$, $G_v(x,u,v)$ are given Carath\'eo\-do\-ry maps defined on $\Omega \times \R\times \R$
which are partial derivatives of a function $G(x,u,v)$.

We prove that, even if the coefficients make the variational approach more difficult, 
under suitable hypotheses functional $\J$, related to problem $(P)$,
admits at least one critical point in the ``right'' Banach space $X$. 
Moreover, if $\J$ is even, then $(P)$ has infinitely many weak bounded solutions.    

The proof, which exploits the interaction between two different norms, 
is based on a weak version of the Cerami--Palais--Smale condition, a
``good'' decomposition of the Banach space $X$  
and suitable generalizations of the Ambrosetti--Rabinowitz Mountain Pass Theorems.
\end{abstract}



\section{Introduction} \label{secintroduction}

In this paper we investigate the existence of one or more solutions for the coupled quasilinear
elliptic system with homogeneous Dirichlet boundary conditions
\begin{equation}   \label{euler}
\left\{
\begin{array}{ll}
- {\rm div} (A(x, u)\vert\nabla u\vert^{p_1 -2} \nabla u) 
+ \frac{1}{p_1}A_u (x, u)\vert\nabla u\vert^{p_1} = G_u(x, u, v) &\hbox{ in $\Omega$,}\\ [10pt]
- {\rm div} (B(x, v)\vert\nabla v\vert^{p_2 -2} \nabla v) 
+\frac{1}{p_2}B_v(x, v)\vert\nabla v\vert^{p_2} = G_v\left(x, u, v\right) &\hbox{ in $\Omega$,}\\[10pt]
u = v = 0 &\hbox{ on $\partial\Omega$,}
\end{array}
\right.
\end{equation}
where $\Omega$ is an open bounded domain in $\R^N$, $N \ge 2$, $p_1, p_2 >1$, 
functions $A$, $B :\Omega \times \R \to \R$ admit partial derivatives
\begin{equation}      \label{AuAv}
A_u(x,u) =\ \frac{\partial A}{\partial u}(x,u), 
\quad  B_v(x,v) =\ \frac{\partial B}{\partial v}(x,v)\qquad
\hbox{for a.e. $x\in \Omega$, for all $u\in \R$, $v\in \R$,}
\end{equation}
and a function $G : \Omega \times \R\times \R \to \R$ exists
such that $\nabla G(x,u,v) = (G_u(x,u,v),G_v(x,u,v))$, i.e.
\begin{equation}     \label{GuGv} 
\frac{\partial G}{\partial u}(x,u,v) = G_u(x,u,v),
\quad \frac{\partial G}{\partial v}(x,u,v) = G_v(x,u,v)\quad
\hbox{for a.e. $x\in \Omega$, for all $(u,v)\in \R^2$.}
\end{equation}

We note that the coefficients $A(x,u)$ and $B(x,v)$ make the variational approach more difficult.
In fact, also investigating the existence of solutions for just an equation
\begin{equation}   \label{euler_eq}
\left\{
\begin{array}{ll}
- {\rm div} (A(x,u) |\nabla u|^{p-2} \nabla u) + \frac1p\ A_u(x,u)
 |\nabla u|^p   =\  g(x,u) &\hbox{in $\Omega$,}\\
u = 0 &\hbox{ on $\partial\Omega$,}
\end{array}
\right.
\end{equation}
requires suitable approaches such as nonsmooth techniques 
or null G\^ateaux derivative only along ``good'' directions
or a suitable variational setting
(see, e.g., \cite{AB1,CP1,CP2,Ca,FMMT,PeSq,Sq}).

In the particular setting $A(x,u) = A^*(x)$ and $B(x,v) = B^*(x)$,
problem \eqref{euler} reduces to
\begin{equation} \label{P01}
\left\{
\begin{array}{ll}
- {\rm div} (A^*(x)\vert\nabla u\vert^{p_1 -2} \nabla u) = G_u(x, u, v) &\hbox{ in $\Omega$,}\\ [10pt]
- {\rm div} (B^*(x)\vert\nabla v\vert^{p_2 -2} \nabla v) = G_v\left(x, u, v\right) &\hbox{ in $\Omega$,}\\[10pt]
u = v = 0 &\hbox{ on $\partial\Omega$,}
\end{array}
\right.
\end{equation}
with $A^*$, $B^* :\Omega \to \R$. If $A^*(x)$, $B^*(x)$ 
are measurable, bounded and strictly 
positive functions, then \eqref{P01} generalizes
the classical $(p_1,p_2)$--Laplacian system 
\begin{equation} \label{P1}
\left\{\begin{array}{ll}
- \Delta_{p_1}\ u \ = \ G_u(x,u,v)
&\hbox{in $\Omega$,}\\
- \Delta_{p_2}\ v \ =\ G_v(x,u,v)&\hbox{in $\Omega$,}\\
u = v = 0 &\hbox{on $\partial \Omega$.}
\end{array}
\right.
\end{equation}

The recent literature provides many works concerning the study of 
coupled elliptic systems of gradient type.
In fact, many authors have studied problem \eqref{P1}
and have obtained several existence results under hypotheses of sublinear,
superlinear, and resonant type on the nonlinearity $G(x,u,v)$ (see, e.g., \cite{AMS, BdF, DSZ, PAO, VdT}
and, if $p_1=p_2=p$, eventually $p=2$, the survey \cite{dF} and references therein).

On the other hand, a generalization of system \eqref{P01} has been studied
in \cite{CMPP} by using a cohomological
local splitting while quasilinear elliptic systems similar to 
\eqref{euler} are studied in \cite{BB} by means of an approximation approach
and in \cite{AG,Sq} via nonsmooth techniques.

Here, we want to extend the existence result in \cite[Theorem 3]{BdF} to our 
quasilinear problem \eqref{euler}, but with stronger subcritical growth conditions
which allow us to find bounded solutions. Moreover, we want
also to state a multiplicity result in hypothesis of symmetry.
 
To this aim, following the approach
developed in \cite{CP1,CP2}, we state enough conditions for recognizing 
the variational structure of problem \eqref{euler} (see Proposition \ref{smooth1}).
Thus, investigating for solutions of \eqref{euler}
reduces to looking for critical points of the nonlinear functional
\begin{equation}      \label{functional}
\J(u,v) = \frac{1}{p_1}\int_{\Omega} A(x,u) \vert\nabla u\vert^{p_1} dx 
+ \frac{1}{p_2}\int_{\Omega} B(x,v) \vert\nabla v\vert^{p_2} dx 
- \int_{\Omega} G(x,u,v) dx
\end{equation}
in the Banach space $X= X_1\times X_2$, 
where $X_i = W^{1,p_i}_0(\Omega) \cap L^\infty(\Omega)$, for $i=1, 2$.

Unluckily, in our setting classical abstract theorems such as the statements in \cite{AR,BBF} 
cannot be applied, since functional $\J$ does not satisfy neither the Palais--Smale condition 
nor its Cerami's variant;
in fact, already in the simpler case
\eqref{euler_eq} a Palais--Smale sequence converging 
in the $W_0^{1,p}(\Omega)$ norm but unbounded in $L^{\infty}(\Omega)$ can be found
(see, for example, \cite[Example 4.3]{CP2017}).
Therefore, by exploiting the interaction between two different norms on $X$,
we introduce the weak Cerami--Palais--Smale condition (see Definition \ref{wCPS})
and apply the abstract theorems in \cite{CP3} in order to prove the existence 
of at least one solution (see Theorem \ref{ThExist}) or
of infinitely many distinct ones if $\J$ is even
(see Theorem \ref{ThMolt} and Corollary \ref{ThMolt1}). The multiplicity result 
is based on a ``good'' decomposition of the Sobolev spaces $W^{1,p_i}_0(\Omega)$  
as given in \cite[Section 5]{CP2}.

In this paper we assume that the coefficients $A(x,u)$ and $B(x,v)$
are $C^1$--Carath\'eodory, bounded on bounded sets, greater of a strictly positive constant
and interact properly with their partial derivatives 
(interactions trivially satisfied by $A^*(x)$ and $B^*(x)$ in problem
\eqref{P01}),
while $G(x,u,v)$ is a $C^1$--Carath\'eodory function
which has a suitable subcritical growth and satisfies an
Ambrosetti--Rabinowitz type condition.
Moreover, in order to state an existence result (see Theorem \ref{ThExist}),
we assume a condition on the behavior of $G(x,u,v)$ at $(0,0)$ which imposes
a growth of $G(x,u,v)$ below a constant of the first straight lines 
associated to the  $-\Delta_{p_1}$ and $-\Delta_{p_2}$
so that functional $\J$ is greater than a positive constant (and thus greater than $\J(0,0)$) 
in a sphere of small radius as required for applying the Mountain Pass Theorem 
(see Theorem \ref{mountainpass}). On the other hand, a multiplicity result
can be stated (see Theorem \ref{ThMolt}) by using a ``good'' version
of the Symmetric Mountain Pass Theorem 
(see Theorem \ref{abstract}) once $\J$ is even and 
$G(x,u,v)$ grows fast enough at infinity.
Anyway, in order to not weigh this introduction down
with too many details, we prefer to specify each
hypothesis when required and to state 
our main results at the beginning of Section \ref{sec_main}
(see Theorems \ref{ThExist} and \ref{ThMolt}).
\smallskip

This work is organized as follows.
In Section \ref{abstractsection} we introduce an abstract setting and, in particular,
the weak Cerami--Palais--Smale condition
and some related existence and multiplicity results which generalize 
the Ambrosetti--Rabinowitz Mountain Pass Theorem (see \cite[Theorem 2.2]{Ra1}) 
and its symmetric version (see \cite[Theorem 9.12]{Ra1}).
In Section \ref{variational}, after introducing some first hypotheses for $A(x,u)$, $B(x, v)$ and $G(x,u,v)$, 
we give the variational principle for problem \eqref{euler}. 
Then, in Section \ref{sec_wcps} we prove that $\J$ 
satisfies the weak Cerami--Palais--Smale condition. 
Finally, in Section \ref{sec_main}, a ``good'' decomposition of
$X$ is given and the main results are stated and proved. Moreover, 
some explicit examples are pointed out
(see Corollaries \ref{cor1} and \ref{cor2})


\section{Abstract setting}
\label{abstractsection}

We denote $\N = \{1, 2, \dots\}$ and, throughout this section, we assume that:
\begin{itemize}
\item $(X, \|\cdot\|_X)$ is a Banach space with dual 
$(X',\|\cdot\|_{X'})$;
\item $(W,\|\cdot\|_W)$ is a Banach space such that
$X \hookrightarrow W$ continuously, i.e. $X \subset W$ and a constant $\sigma_0 > 0$ exists
such that
\[
\|\xi\|_W \ \le \ \sigma_0\ \|\xi\|_X\qquad \hbox{for all $\xi \in X$;}
\]
\item $J : {\cal D} \subset W \to \R$ and $J \in C^1(X,\R)$ with $X \subset {\cal D}$;
\item $K_J = \{ \xi \in X:\ dJ(\xi) = 0\}$ is
the set of the critical points of $J$ in $X$.
\end{itemize}
Furthermore, fixing $\beta \in \R$, we denote as
\begin{itemize}
\item $K_J^\beta = \{ \xi \in X:\ J(\xi) = \beta,\ dJ(\xi) = 0\}$
the set of the critical points of $J$ in $X$ at level $\beta$;
\item $J^\beta = \{ \xi\in X:\ J(\xi) \leq \beta\}$
the sublevel of $J$ with respect to $\beta$.
\end{itemize}

Anyway, in order to avoid any
ambiguity and simplify, when possible, the notation, 
from now on by $X$ we denote the space equipped with
its given norm $\|\cdot\|_X$ while, if the norm $\Vert\cdot\Vert_{W}$ is involved,
we write it explicitly.

For simplicity, taking $\beta \in \R$, we say that a sequence
$(\xi_n)_n\subset X$ is a {\sl Cerami--Palais--Smale sequence at level $\beta$},
briefly {\sl $(CPS)_\beta$--sequence}, if
\[
\lim_{n \to +\infty}J(\xi_n) = \beta\quad\mbox{and}\quad 
\lim_{n \to +\infty}\|dJ\left(\xi_n\right)\|_{X'} (1 + \|\xi_n\|_X) = 0.
\]
Moreover, $\beta$ is a {\sl Cerami--Palais--Smale level}, briefly {\sl $(CPS)$--level}, 
if there exists a $(CPS)_\beta$--sequence.

As $(CPS)_\beta$--sequences may exist which are unbounded in $\|\cdot\|_X$
but converge with respect to $\|\cdot\|_W$,
we have to weaken the classical Cerami--Palais--Smale 
condition in a suitable way according to the ideas already developed in 
previous papers (see, e.g., \cite{CP1,CP2,CP3}).  

\begin{definition} \label{wCPS}
The functional $J$ satisfies the
{\slshape weak Cerami--Palais--Smale 
condition at level $\beta$} ($\beta \in \R$), 
briefly {\sl $(wCPS)_\beta$ condition}, if for every $(CPS)_\beta$--sequence $(\xi_n)_n$,
a point $\xi \in X$ exists, such that 
\begin{description}{}{}
\item[{\sl (i)}] $\displaystyle 
\lim_{n \to+\infty} \|\xi_n - \xi\|_W = 0\quad$ (up to subsequences),
\item[{\sl (ii)}] $J(\xi) = \beta$, $\; dJ(\xi) = 0$.
\end{description}
If $J$ satisfies the $(wCPS)_\beta$ condition at each level $\beta \in I$, $I$ real interval, 
we say that $J$ satisfies the $(wCPS)$ condition in $I$.
\end{definition}

Since in \cite{CP3} a Deformation Lemma has been proved 
if the functional $J$ satisfies a weaker version of the $(wCPS)_\beta$ condition,
namely any $(CPS)$--level is also a critical level, in particular we can state the following result. 

\begin{lemma}[Deformation Lemma] \label{def}
Let $J\in C^1(X,\R)$ and consider $\beta \in \R$ such that
\begin{itemize}
\item $J$ satisfies the $(wCPS)_\beta$ condition, 
\item $\ K_J^\beta = \emptyset$.
\end{itemize}
Then, fixing any $\bar \eps > 0$, there exist
a constant $\eps > 0$ and a homeomorphism $h_\eps : X \to X$
such that $2\eps < {\bar \eps}$ and
\begin{itemize}
\item[$(i)$] $h_\eps(J^{\beta+\eps}) \subset J^{\beta-\eps}$,
\item[$(ii)$] $h_\eps(\xi) = \xi$ for all $\xi \in X$ such that either
$J(\xi) \leq \beta-\bar\eps$ or $J(\xi) \ge \beta+\bar\eps$.
\end{itemize}
Moreover, if $J$ is even on $X$, then $h_\eps$ can be chosen odd.
\end{lemma}

\begin{proof}
It is enough reasoning as in \cite[Lemma 2.3]{CP3} with $\beta_1 = \beta_2 = \beta$
and pointing out that the deformation $h_\eps: X \to X$ is a homeomorphism.    
\end{proof}

From Lemma \ref{def} we obtain the following generalization of the Mountain Pass Theorem
(compare it with \cite[Theorem 1.7]{CP3} and the classical statement in \cite[Theorem 2.2]{Ra1}).

\begin{theorem}
\label{mountainpass}
Let $J\in C^1(X,\R)$ be such that $J(0) = 0$
and the $(wCPS)$ condition holds in $\R_+$.\\
Moreover, assume that there exist some constants $R_0$, $\varrho_0 > 0$, and  
 $e \in X$ such that 
\begin{itemize}
\item[$(i)$] $\; \xi \in X, \quad \|\xi\|_W = R_0\qquad \then\qquad J(\xi) \geq \varrho_0$;
\item[$(ii)$] $\; \| e\|_W > R_0\qquad\hbox{and}\qquad J(e) < \varrho_0$.
\end{itemize}
Then, $J$ has a Mountain Pass critical point $\xi \in X$ such that $J(\xi) \geq \varrho_0$.
\end{theorem}

Furthermore, with the stronger assumption that $J$ is symmetric, 
 the following generalization of the symmetric Mountain Pass Theorem can be stated
(see \cite[Theorem 2.4]{CPS_CCM} and compare with \cite[Theorem 9.12]{Ra1} and 
\cite[Theorem 2.4]{BBF}).

\begin{theorem}
\label{abstract}
Let $J\in C^1(X,\R)$ be an even functional such that $J(0) = 0$ and
the $(wCPS)$ condition holds in $\R_+$.
Moreover, assume that $\varrho > 0$ exists so that:
\begin{itemize}
\item[$({\cal H}_{\varrho})$]
three closed subsets $V_\varrho$, $Y_\varrho$ and $\M_\varrho$ of $X$ and a constant
$R_\varrho > 0$ exist which satisfy the following conditions:
\begin{itemize}
\item[$(i)$] $V_\varrho$ and $Y_\varrho$ are subspaces of $X$ such that
\[
V_\varrho + Y_\varrho = X,\qquad \codim Y_\varrho\ <\ \dim V_\varrho\ <\ +\infty;
\]
\item[$(ii)$] $\M_\varrho = \partial \enne$, where $\enne \subset X$ is a neighborhood of the origin
which is symmetric and bounded with respect to $\|\cdot\|_W$; 
\item[$(iii)$]  $\ \xi \in \M_\varrho \cap Y_\varrho\qquad \then\qquad J(\xi) \ge \varrho$;
\item[$(iv)$] $\ \xi \in V_\varrho, \quad \|\xi\|_X \ge R_\varrho \qquad \then\qquad J(\xi) \le 0$.
\end{itemize}
\end{itemize}
Then, if we put 
\[
\beta_\varrho\ =\ \inf_{\gamma \in \Gamma_\varrho} \sup_{\xi\in V_\varrho} J(\gamma(\xi)),
\]
with 
\[
\Gamma_\varrho\ =\ \{\gamma : X \to X:\
\gamma\ \hbox{odd homeomeorphism,}\quad \gamma(\xi) = \xi \ 
\hbox{if $\xi \in V_\varrho$ with $\|\xi\|_X \ge R_\varrho$}\},
\]
the functional $J$ possesses at least a pair of symmetric critical points in $X$ 
with corresponding critical level $\beta_\varrho$ which belongs to $[\varrho,\varrho_1]$,
where $\varrho_1 \ge \displaystyle \sup_{\xi \in V_\varrho}J(\xi) > \varrho$.
\end{theorem}

If we can apply infinitely many times Theorem \ref{abstract}, then the following multiplicity abstract 
result can be stated.

\begin{corollary}
\label{multiple}
Let $J \in C^1(X,\R)$ be an even functional such that $J(0) = 0$,
the $(wCPS)$ condition holds in $\R_+$ and a sequence $(\varrho_n)_n \subset\ ]0,+\infty[$ exists
such that $\varrho_n \nearrow +\infty$ and  
assumption $({\cal H}_{\varrho_n})$ holds for all $n \in \N$.\\
Then, functional $J$ possesses a sequence of critical points $(u_{k_n})_n \subset X$ such that
$J(u_{k_n}) \nearrow +\infty$ as $n \nearrow +\infty$.
\end{corollary}


\section{Variational setting and first properties}
\label{variational}

From now on, let $\Omega \subset \R^N$ be an open bounded domain, $N\ge 2$,
so we denote by:
\begin{itemize}
\item $L^r(\Omega)$ the Lebesgue space with
norm $|\xi|_r = \left(\int_\Omega|\xi|^r dx\right)^{1/r}$ if $1 \le r < +\infty$;
\item $L^\infty(\Omega)$ the space of Lebesgue--measurable 
and essentially bounded functions $\xi :\Omega \to \R$ with norm $\displaystyle |\xi|_{\infty} = \esssup_{\Omega} |\xi|$;
\item $W^{1,p}_0(\Omega)$ the classical Sobolev space with
norm $\|\xi\|_{W_0^{1, p}} = |\nabla \xi|_{p}$ if $1 \le p < +\infty$;
\item $\meas(D)$ the usual Lebesgue measure of a measurable set $D$ in $\R^N$.
\end{itemize}

For simplicity, here and in the following we denote by $|\cdot|$  
the standard norm on any Euclidean space, as the dimension
of the considered vector is clear and no ambiguity occurs, and 
by $C$ any strictly positive constant which arises by computation. 

In order to investigate the existence of weak solutions  
of the nonlinear problem \eqref{euler},
consider $p_1$, $p_2 > 1$ and, for $ i\in\lbrace 1, 2\rbrace$, the related Sobolev space
\[
W_i = W_0^{1, p_i}(\Omega)\quad \hbox{with norm $\Vert \cdot\Vert_{W_i} = \Vert \cdot\Vert_{W_0^{1, p_i}}$.}
\]
From the Sobolev Embedding Theorem, for any $r\in[1, p_i^{\ast}]$ 
with $p_i^{\ast} = \frac{Np_i}{N-p_i}$ if $N>p_i$, 
or any $r\in[1, +\infty[$ if $p_i \ge N$, $W_i$ is
continuously embedded in $L^{r}(\Omega)$, i.e. a positive constant 
$\tau_{i,r}$ exists such that
\begin{equation}   \label{Sobpi}
\vert \xi\vert_r \leq\tau_{i,r}\Vert \xi\Vert_{W_i} \quad \mbox{ for all } \xi\in W_i.
\end{equation}
For simplicity, we put
\[
p^*_i = +\infty\quad \hbox{and} \quad \frac{1}{p^*_i} = 0
\qquad \hbox{if $p_i \ge N$.}
\]

\begin{definition}
A function $f:\Omega\times\R^m\to\R$, $m \in \N$, is a $\mathcal{C}^{k}$--Carath\'eodory function, 
$k\in\N\cup\lbrace 0\rbrace$, if
\begin{itemize}
\item $f(\cdot, \nu) : x \in \Omega \mapsto f(x, \nu) \in \R$ is measurable for all $\nu \in \R^m$,
\item $f(x,\cdot) : \nu \in \R^m \mapsto f(x, \nu) \in \R$ is $\mathcal{C}^k$ for a.e. $x \in \Omega$.
\end{itemize}
\end{definition}

Let $A:(x,u) \in \Omega\times\R \mapsto A(x,u) \in \R$ and $B:(x,v) \in \Omega\times\R \mapsto B(x,v)\in \R$
be two given functions and assume that $G : \Omega \times \R\times \R \to \R$ exists
such that the following conditions hold:
\begin{itemize}
\item[$(h_0)$]
$A$ and $B$ are $\mathcal{C}^1$--Carath\'eodory functions with partial derivatives as in \eqref{AuAv};
\item[$(h_1)$] for any $\rho > 0$ we have that
\begin{align*}
&\sup_{\vert u\vert\leq \rho} \vert A\left(\cdot, u\right)\vert \in L^{\infty}\left(\Omega\right), \ 
\quad \sup_{\vert v\vert\leq \rho} \vert B\left(\cdot, v\right)\vert \in L^{\infty}\left(\Omega\right),\\
&\sup_{\vert u\vert\leq \rho} \vert A_u\left(\cdot, u\right)\vert \in L^{\infty}\left(\Omega\right), 
\quad \sup_{\vert v\vert\leq \rho} \vert B_v\left(\cdot, v\right)\vert \in L^{\infty}\left(\Omega\right);
\end{align*}
\item[$(g_0)$]
$G$ is a $\mathcal{C}^1$--Caratheodory function with partial derivatives as in \eqref{GuGv} 
such that
\[
G(\cdot, 0, 0) \in L^\infty(\Omega)
\] 
and
\[
G_u(x, 0, 0) = G_v(x, 0, 0) = 0 \quad \mbox{ for a.e. } x\in\Omega;
\] 
\item[$(g_1)$] some real numbers $q_i \geq 1$,
$s_i \geq 0$ if $i \in \{1,2\}$, and $\sigma >0$ exist such that
\begin{eqnarray}      \label{gucr}
&&\vert G_u(x,u,v)\vert\leq \sigma(1 +\vert u\vert^{q_1 -1} +\vert v\vert^{s_1})
\quad \mbox{ for a.e. $x\in\Omega$, all $(u,v) \in \R^2$,}\\
\label{gvcr}
&&\vert G_v(x,u,v)\vert\leq \sigma(1 +\vert u\vert^{s_2} +\vert v\vert^{q_2 -1})
\quad \mbox{ for a.e. $x\in\Omega$, all $(u,v) \in \R^2$,}
\end{eqnarray}
with
\begin{equation}    \label{crit_exp}
1 \le q_1 < p_1^{\ast},\qquad 1 \le q_2 < p_2^{\ast},
\end{equation}
and
\begin{equation}    \label{crit_expi}
0 \le s_1 < \frac{p_1}{N} \left(1 - \frac{1}{p^*_1}\right) p_2^*, \qquad
0 \le s_2 < \frac{p_2}{N} \left(1 - \frac{1}{p^*_2}\right) p_1^*.
\end{equation}
\end{itemize}

\begin{remark}
We note that the subcritical growth assumptions \eqref{crit_exp} and \eqref{crit_expi}
are required in order to prove the $(wCPS)$ condition (see Proposition \ref{PropwCPS})
but not the variational principle stated in Proposition \ref{smooth1}.\\
Moreover, conditions $(g_0)$, \eqref{gucr} and \eqref{gvcr} imply the
existence of a constant $\sigma_1 > 0$ such that
\begin{equation}    \label{Gsigma}
\vert G(x,u,v)\vert\leq \sigma_1\left(1 + \vert u\vert +\vert u\vert^{q_1} + \vert u\vert \vert v\vert^{s_1} 
+ \vert v\vert +\vert u\vert^{s_2}\vert v\vert + \vert v\vert^{q_2}\right)
\end{equation}
for a.e. $x\in\Omega$, all $(u,v)\in \R^2$.\\
In fact, from $(g_0)$, \eqref{gucr} and \eqref{gvcr}, 
for a.e. $x\in\Omega$ and for any $(u, v)\in\R^2$ 
the Mean Value Theorem implies the existence of $t\in\left] 0, 1\right[$ such that
\begin{align*}
\vert G(x,u,v)\vert &\le \vert G(x,u,v) - G(x,0,0)\vert + \vert G(x,0,0)\vert 
= \vert G_u(x, tu, tv) u + G_v(x, tu, tv) v\vert + \vert G(x,0,0)\vert\\
&\leq\sigma\left( 1 +\vert u\vert^{q_1 -1} +\vert v\vert^{s_1}\right)\vert u\vert 
+\sigma\left(1 +\vert u\vert^{s_2} +\vert v\vert^{q_2 -1}\right)\vert v\vert
+ \vert G(\cdot,0,0)\vert_\infty
\end{align*}
and so \eqref{Gsigma} follows with $\sigma_1 = \max\{\sigma,\vert G(\cdot,0,0)\vert_\infty\}$.
\end{remark}

Here, the notation introduced for the abstract
setting at the beginning of Section \ref{abstractsection}
is referred to our problem with 
\begin{equation}     \label{Wdefn1}
W = W_1\times W_2
\end{equation}
while the Banach space $(X,\|\cdot\|_X)$ is defined as 
\begin{equation}     \label{Xdefn1}
X = X_1\times X_2, 
\end{equation}
where
\[
X_1:= W_1\cap L^{\infty}(\Omega) \quad \mbox{ and }\quad X_2:= W_2\cap L^{\infty}(\Omega)
\]
with the norms
\begin{equation}         \label{Xnorms}
\Vert u\Vert_{X_1} = \Vert u\Vert_{W_1} + \vert u\vert_{\infty} \mbox{ if } u\in X_1 \quad 
\mbox{ and } \quad \Vert v\Vert_{X_2} = \Vert v\Vert_{W_2} + \vert v\vert_{\infty} \ \mbox{ if } v\in X_2.
\end{equation}
Since $(W_i, \Vert\cdot\Vert_{W_i})$ is a reflexive Banach space
for both $i=1$ and $i=2$, 
so is $\left(W, \Vert\cdot\Vert_W\right)$ where, for $(u, v)\in W$ it is
\[
\Vert(u, v)\Vert_W = \Vert u\Vert_{W_1} + \Vert v\Vert_{W_2}.
\]
Setting 
\[
L:= L^{\infty}(\Omega)\times L^{\infty}(\Omega)\quad \hbox{with}\; 
\Vert (u, v)\Vert_{L} = \vert u\vert_{\infty} + \vert v\vert_{\infty},
\]
we have that $X$ in \eqref{Xdefn1} can also be written as 
\[
X = W\cap L 
\]
and can be equipped with the norm
\[
\Vert(u, v)\Vert_X = \Vert (u, v)\Vert_W + \Vert (u, v)\Vert_L 
= \Vert u\Vert_{X_1} + \Vert v\Vert_{X_2}.
\]
By definition, for $i \in \{1, 2\}$ we have $X_i \hookrightarrow W_i$ and $X_i \hookrightarrow L^\infty(\Omega)$
with continuous embeddings. 

\begin{remark}
If  $i \in \{1, 2\}$ is such that $p_i > N$, then $X_i = W_i$, as $W_i \hookrightarrow L^\infty(\Omega)$.
Hence, if both $p_1 > N$ and $p_2 > N$, 
then $X = W_1 \times W_2$ and
the classical Mountain Pass Theorems in \cite{AR} can be used, if required.
\end{remark}

Now, if conditions $(h_0)$--$(h_1)$, $(g_0)$, \eqref{gucr} and \eqref{gvcr} hold,
by direct computations it follows that $\J(u,v)$ in \eqref{functional}
is well defined for all $(u, v)\in X$. Moreover,
taking any $(u, v)$, $(w, z)\in X$, the G\^ateaux differential 
of functional $\J$ in $(u,v)$ along the direction $(w,z)$ is
\begin{equation}     \label{diff}
\begin{split}
d\J(u, v)[(w, z)] = &\int_{\Omega} A(x, u)\vert\nabla u\vert^{p_1 -2}\nabla u\cdot \nabla w \ dx + 
\frac{1}{p_1}\int_{\Omega} A_u(x, u) w \vert\nabla u\vert^{p_1} dx \\    
&+ \int_{\Omega} B(x, v)\vert\nabla v\vert^{p_2 -2} \nabla v\cdot \nabla z \ dx+ 
\frac{1}{p_2}\int_{\Omega} B_v(x, v) z \vert\nabla v\vert^{p_2}dx\\     
&-\int_{\Omega} G_u(x, u, v) w \ dx - \int_{\Omega} G_v(x, u, v) z \ dx.
\end{split}
\end{equation}
For simplicity, we put
\[
\begin{split}
&\frac{\partial\J}{\partial u}(u, v): w\in X_1\mapsto \frac{\partial\J}{\partial u}(u, v)[w] = d\J(u, v)[(w, 0)]\in\R, \\
&\frac{\partial\J}{\partial v}(u, v): z\in X_2\mapsto \frac{\partial\J}{\partial v}(u, v)[z] = d\J(u, v)[(0, z)]\in\R.
\end{split}
\]
So, from \eqref{diff} it follows that
\begin{equation}     \label{dJw}
\begin{split}
\frac{\partial\J}{\partial u}(u, v) [w] = &\int_{\Omega} A(x, u)\vert\nabla u\vert^{p_1 -2} \nabla u\cdot\nabla w\ dx 
+ \frac{1}{p_1}\int_{\Omega} A_u(x, u) w \vert\nabla u\vert^{p_1} dx\\
&- \int_{\Omega} G_u(x, u, v) w\ dx
\end{split}
\end{equation}
and
\begin{equation}    \label{dJz}
\begin{split}
\frac{\partial\J}{\partial v}(u, v) [z] = 
&\int_{\Omega} B(x, v)\vert\nabla v\vert^{p_2 -2} \nabla v\cdot\nabla z\ dx 
+ \frac{1}{p_2}\int_{\Omega} B_v(x, v) z \vert\nabla v\vert^{p_2} dx\\
&- \int_{\Omega} G_v(x, u, v) z dx.
\end{split}
\end{equation}

\begin{remark}
Taking $(u, v)\in X$, since $d\J(u, v)\in X^{\prime}$, 
then 
\[
\frac{\partial\J}{\partial u}(u, v)\in X_1^{\prime}, \qquad
\frac{\partial\J}{\partial v}(u, v)\in X_2^{\prime}
\]
and
\begin{equation}    \label{dJz1}
d\J(u, v)[(w, z)] = \frac{\partial\J}{\partial u}(u, v)[ w] 
+ \frac{\partial\J}{\partial v}(u, v)[z] \quad \forall (w, z)\in X.
\end{equation}
Moreover, direct computations imply that
\begin{equation}    \label{dJzstar}
\left\|\frac{\partial\J}{\partial u}(u, v)\right\|_{X_1'} \le \|d\J(u, v)\|_{X'},\quad
\left\|\frac{\partial\J}{\partial v}(u, v)\right\|_{X_2'} \le \|d\J(u, v)\|_{X'},  
\end{equation}
and 
\begin{equation}    \label{dJbis}
\|d\J(u, v)\|_{X'} \le \left\|\frac{\partial\J}{\partial u}(u, v)\right\|_{X_1'} +
\left\|\frac{\partial\J}{\partial v}(u, v)\right\|_{X_2'}.
\end{equation}
Clearly, we have
\[
d\J(u, v) = 0 \;\hbox{in $X$}\quad\iff\quad
\frac{\partial\J}{\partial u}(u, v)= 0 \;\hbox{in $X_1$}
\;\hbox{and}\; \frac{\partial\J}{\partial v}(u,v) = 0 \;\hbox{in $X_2$.} 
\]
\end{remark}

As useful in the following, we recall this technical lemma (see \cite{GL}).

\begin{lemma}      \label{lemmavett}
A constant $C > 0$ exists such that for any $\xi, \eta\in\R^N, N\geq 1$, it results
\[
\begin{split}
&\vert \vert \xi\vert^{r-2} \xi -\vert \eta\vert^{r-2} \eta\vert\ \leq\ 
C \vert \xi -\eta\vert\left(\vert \xi\vert +\vert \eta\vert\right)^{r-2} \qquad\hbox{ if } r > 2,\\
&\vert \vert \xi\vert^{r-2} \xi -\vert \eta\vert^{r-2} \eta\vert
\leq C \vert \xi -\eta\vert^{r-1} \ \qquad \qquad\qquad\hbox{ if } 1< r\leq 2.
\end{split}
\]
\end{lemma}

From Lemma \ref{lemmavett} and reasoning as in \cite[Lemma 3.4]{CS2020}, we have the following estimate.

\begin{lemma}   \label{lemma2020}
If $p>1$, a constant $C = C(p) >0$ exists such that
\[
\int_{\Omega} \vert \vert\nabla w\vert^p -\vert\nabla z\vert^p\vert dx 
\leq C \Vert w-z\Vert_{W_0^{1, p}}\left(\Vert w\Vert_{W_0^{1, p}}^{p-1} 
+\Vert z\Vert_{W_0^{1, p}}^{p-1}\right) \quad \forall w, z \in W_0^{1, p}(\Omega).
\]
\end{lemma}

Now, we can state a regularity result.

\begin{proposition}\label{smooth1}
Assume that conditions $(h_0)$--$(h_1)$, $(g_0)$, \eqref{gucr} and \eqref{gvcr} hold. 
Let $((u_n, v_n))_n \subset X$ and $(u, v) \in X$ be such that
\begin{align}     \label{unconv}
&u_n\to u \mbox{ in } W_1, \quad u_n\to u \mbox{ a.e. in } \Omega \mbox{ if } n\to\infty,\\    \label{vnconv}
&v_n\to v \mbox{ in } W_2, \quad v_n\to v \mbox{ a.e. in } \Omega \mbox{ if } n\to\infty, 
\end{align}
and, moreover, $M > 0$ exists such that
\begin{equation}       \label{unvnuniflim}
\vert u_n\vert_{\infty}\leq M \quad \mbox{ and } \quad \vert v_n\vert_{\infty}\leq M \quad \forall n\in\N.
\end{equation}
Then,
\[
\J(u_n, v_n) \to \J(u, v) \quad \mbox{ and } \quad 
\Vert d\J(u_n, v_n) - d\J(u, v)\Vert_{X^{\prime}}\to 0 \ \mbox{ as } n\to\infty.
\]
Hence, $\J$ is a $C^1$ functional on $X$ with Fr\'echet differential  
defined as in (\ref{diff}).
\end{proposition}

\begin{proof}
Let $((u_n, v_n))_n\subset X$ and $(u,v)\in X$ be such 
that conditions \eqref{unconv}--\eqref{unvnuniflim} hold. Firstly, we prove that
\begin{equation}      \label{JtoJ}
\J\left(u_n, v_n\right)\longrightarrow \J\left(u, v\right) \quad \mbox{ if } n\to +\infty.
\end{equation}
To this aim, we note that
\begin{align*}
&\vert \J(u_n, v_n) - \J(u, v)\vert\ \leq\ 
\frac{1}{p_1}\int_{\Omega} \vert A(x, u_n)\vert\ \vert \vert\nabla u_n\vert^{p_1} -\vert\nabla u\vert^{p_1}\vert dx\\
&\qquad+ \frac{1}{p_1}\int_{\Omega}\left\vert A(x, u_n) - A(x, u)\right\vert \left\vert\nabla u\right\vert^{p_1}dx
+ \frac{1}{p_2}\int_{\Omega}\vert B(x, v_n)\vert\left\vert\left\vert\nabla v_n\right\vert^{p_2} 
-\left\vert\nabla v\right\vert^{p_2}\right\vert dx\\
&\qquad + \frac{1}{p_2}\int_{\Omega} \left\vert B(x, v_n) - B(x, v)\right\vert\left\vert\nabla v\right\vert^{p_2} dx
+ \int_{\Omega}\left\vert G(x, u_n, v_n) - G(x, u, v)\right\vert dx.
\end{align*}
From $(h_1)$, \eqref{unconv}, \eqref{unvnuniflim} and Lemma \ref{lemma2020} it follows that
\begin{equation}         \label{Aprima}
\int_{\Omega} \vert A(x, u_n)\vert \left\vert\left\vert\nabla u_n\right\vert^{p_1} - 
\left\vert\nabla u\right\vert^{p_1}\right\vert dx \longrightarrow 0.
\end{equation}
On the other hand, from Dominated Convergence Theorem
we have that
\begin{equation}      \label{Aseconda}
\int_{\Omega} \left\vert A(x, u_n) - A(x, u)\right\vert \left\vert\nabla u\right\vert^{p_1} dx \longrightarrow 0,
\end{equation}
as \eqref{unconv} and $(h_0)$ imply that 
\[
\left\vert A(x, u_n) - A(x, u)\right\vert\left\vert\nabla u\right\vert^{p_1}\longrightarrow 0 \quad \mbox{ a.e. in $\Omega$}
\]
and from \eqref{unvnuniflim} and $(h_1)$ it results
\begin{align*}
\left\vert A(x, u_n) - A(x, u)\right\vert\left\vert\nabla u\right\vert^{p_1} \leq C\left\vert\nabla u\right\vert^{p_1}\in L^1(\Omega).
\end{align*}
Similarly, by assumptions $(h_0)$--$(h_1)$, \eqref{vnconv}, \eqref{unvnuniflim} and again Lemma \ref{lemma2020}, we obtain that
\begin{align*}
\int_{\Omega} \vert B(x, v_n)\vert \left\vert\left\vert\nabla v_n\right\vert^{p_2} -
\left\vert\nabla v\right\vert^{p_2}\right\vert dx \to 0, \quad 
\int_{\Omega}\left\vert B(x, v_n) - B(x, v)\right\vert\left\vert\nabla v\right\vert^{p_2}dx \to 0.
\end{align*}
Finally, \eqref{unconv}, \eqref{vnconv} and $(g_0)$ imply that
\[
G(x, u_n, v_n)\longrightarrow G(x, u, v) \quad \mbox{ a.e. in $\Omega$,}
\]
while, by conditions \eqref{Gsigma} and \eqref{unvnuniflim}, a constant $C > 0$ exists such that
\[
\left\vert G(x, u_n, v_n)\right\vert\leq C \quad \mbox{ for a.e. } x\in\Omega,
\]
thus, by Dominated Convergence Theorem we conclude that
\[
\int_{\Omega} G(x, u_n, v_n) dx \longrightarrow \int_{\Omega} G(x, u, v) dx
\]
and \eqref{JtoJ} holds.\\
Now, in order to prove that
\[
\|d\J(u_n, v_n) - d\J(u, v)\|_{X'} \longrightarrow 0,
\]
we note that it is enough to prove that
\begin{align}\label{uno1}
&\left\|\frac{\partial\J}{\partial u}(u_n,v_n) - \frac{\partial\J}{\partial u}(u, v)\right\|_{X_1'} \longrightarrow 0,\\
&\left\|\frac{\partial\J}{\partial v}(u_n,v_n) - \frac{\partial\J}{\partial v}(u, v)\right\|_{X_2'} \longrightarrow 0, \label{due2}
\end{align}
as from \eqref{dJz1} it follows that
\[
\|d\J(u_n, v_n) - d\J(u, v)\|_{X'} \le 
\left\|\frac{\partial\J}{\partial u}(u_n,v_n) - \frac{\partial\J}{\partial u}(u, v)\right\|_{X_1'} +
\left\|\frac{\partial\J}{\partial v}(u_n,v_n) - \frac{\partial\J}{\partial v}(u, v)\right\|_{X_2'}.
\]
To this aim, firstly let $w \in X_1$ be such that $\|w\|_{X_1} \le 1$.
So, from \eqref{dJw}, condition $\vert w\vert_{\infty}\leq 1$ implies that
\begin{align*}
&\left\vert\frac{\partial\J}{\partial u}(u_n,v_n)[w] - \frac{\partial\J}{\partial u}(u, v)[w]\right\vert 
\leq \int_{\Omega} \vert A(x, u_n)\vert\ \vert\vert\nabla u_n\vert^{p_1 - 2} \nabla u_n - 
\vert\nabla u\vert^{p_1 - 2} \nabla u\vert\ \vert\nabla w\vert dx\\
&\qquad + \int_{\Omega}\vert A(x, u_n) - A(x, u)\vert\ \vert\nabla u\vert^{p_1 - 1} \vert \nabla w\vert dx\\ 
&\qquad + \frac{1}{p_1} \int_{\Omega}\vert A_{u}(x, u_n)\vert\ \vert\vert\nabla u_n\vert^{p_1} - \vert\nabla u\vert^{p_1}\vert dx
+ \frac{1}{p_1}\int_{\Omega} \vert A_{u}(x, u_n) - A_{u}(x, u)\vert\ \vert\nabla u\vert^{p_1} dx \\
&\qquad +\int_{\Omega}\vert G_u(x, u_n, v_n) - G_u(x, u, v)\vert dx.
\end{align*}
From assumptions $(g_0)$, \eqref{gucr}, \eqref{gvcr}, conditions \eqref{unconv}--\eqref{unvnuniflim} and again Dominated Convergence Theorem 
it follows that
\[
\int_{\Omega} \vert G_{u}(x, u_n, v_n) - G_{u}(x, u, v)\vert dx\longrightarrow 0.
\]
Moreover, with arguments which are similar to those ones used previously up for proving \eqref{Aprima}, 
respectively \eqref{Aseconda}, we have that
\[
\int_{\Omega}\vert A_{u}(x, u_n)\vert\ \vert\vert\nabla u_n\vert^{p_1} - \vert\nabla u\vert^{p_1}\vert dx \longrightarrow 0,
\]
respectively
\[
\int_{\Omega} \vert A_{u}(x, u_n) - A_{u}(x, u)\vert\ \vert\nabla u\vert^{p_1} dx \longrightarrow 0.
\]
Furthermore, since $\Vert w\Vert_{W_1}\leq 1$, by using Hölder inequality we have that
\[
\int_{\Omega}\vert A(x, u_n) - A(x, u)\vert \vert\nabla u\vert^{p_1-1} \vert\nabla w\vert dx
\leq\left(\int_{\Omega}\vert A(x, u_n) - A(x, u)\vert^{\frac{p_1}{p_1 - 1}} \vert\nabla u\vert^{p_1} dx\right)^{\frac{p_1 -1}{p_1}},
\]
where
\[
\int_{\Omega} \vert A(x, u_n) - A(x, u)\vert^{\frac{p_1}{p_1 - 1}} \vert\nabla u\vert^{p_1} dx \longrightarrow 0 
\]
as \eqref{unconv} and hypothesis $(h_0)$ imply
\[
\vert A(x, u_n) - A(x, u)\vert^{\frac{p_1}{p_1 -1}} \vert\nabla u\vert^{p_1}\longrightarrow 0 \quad \mbox{ a.e. in } \Omega
\]
while from \eqref{unvnuniflim} and $(h_1)$ it follows that
\[
\vert A(x, u_n) - A(x, u)\vert^{\frac{p_1}{p_1 -1}} \vert\nabla u\vert^{p_1}\leq C\vert\nabla u\vert^{p_1} \quad \mbox{ a.e. in } \Omega
\]
so Dominated Convergence Theorem applies.\\
At last, hypothesis $(h_1)$, \eqref{unvnuniflim} and Hölder inequality imply that
\begin{equation}\label{gradHol}
\begin{split}
&\int_{\Omega}\vert A(x, u_n)\vert\ \vert \vert\nabla u_n\vert^{p_1 -2}\nabla u_n 
- \vert\nabla u\vert^{p_1 -2}\nabla u\vert\ \vert\nabla w\vert dx\\     
&\quad\qquad\leq C\left(\int_{\Omega} \vert \vert\nabla u_n\vert^{p_1 -2}\nabla u_n - 
\vert\nabla u\vert^{p_1 -2}\nabla u\vert^{\frac{p_1}{p_1 -1}} dx\right)^{\frac{p_1 -1 }{p_1}}.
\end{split}
\end{equation}
We note that if $p_1 > 2$, then Lemma \ref{lemmavett} and again Hölder inequality, 
with conjugate exponents $p_1 -1$ and $\frac{p_1 -1}{p_1 -2}$, give
\begin{equation}     \label{grad1lemma}
\begin{split}
&\left(\int_{\Omega} \vert \vert\nabla u_n\vert^{p_1 -2}\nabla u_n - 
\vert\nabla u\vert^{p_1 -2}\nabla u\vert^{\frac{p_1}{p_1 -1}} dx\right)^{\frac{p_1 -1 }{p_1}}\\
&\qquad\leq \ C\ \left(\int_{\Omega} \vert\nabla u_n -\nabla u\vert^{\frac{p_1}{p_1 -1}}
\left(\vert\nabla u_n\vert +\vert\nabla u\vert\right)^{\frac{p_1(p_1 -2)}{p_1 -1}} dx\right)^{\frac{p_1 -1}{p_1}}\\
&\qquad\leq \ C\ \left(\int_{\Omega}\vert\nabla u_n -\nabla u\vert^{p_1} dx\right)^{\frac{1}{p_1}} 
\left(\int_{\Omega}\left(\vert\nabla u_n\vert +\vert\nabla u\vert\right)^{p_1} dx\right)^{\frac{p_1 -2}{p_1}}\\
&\qquad\leq\ C\ \Vert u_n -u\Vert_{W_1}
\end{split}
\end{equation}
as \eqref{unconv} implies that $(\|u_n\|_{W_1})_n$ is bounded.\\
On the other hand, if $1 <p_1\leq 2$, from Lemma \ref{lemmavett} we have
\begin{equation}    \label{grad2lemma}
\left(\int_{\Omega} \vert \vert\nabla u_n\vert^{p_1 -2}\nabla u_n -
\vert\nabla u\vert^{p_1 -2}\nabla u\vert^{\frac{p_1}{p_1 -1}} dx\right)^{\frac{p_1 -1}{p_1}}
\leq\ C\ \Vert u_n -u\Vert_{W_1}^{p_1 -1}.
\end{equation}
Hence, from \eqref{unconv} and \eqref{gradHol}--\eqref{grad2lemma} it follows that
\[
\int_{\Omega}\vert A(x, u_n)\vert\ \vert \vert\nabla u_n\vert^{p_1 -2}\nabla u_n - 
\vert\nabla u\vert^{p_1 -2}\nabla u\vert\ \vert\nabla w\vert dx\longrightarrow 0 
\quad \mbox{ uniformly with respect to $w$.}
\]
Thus, summing up, we have that
\[
\left\vert\frac{\partial\J}{\partial u}(u_n,v_n)[w] - \frac{\partial\J}{\partial u}(u, v)[w]\right\vert \longrightarrow 0 
\quad \mbox{ uniformly with respect to $w \in W_1$, $\|w\|_{X_1} \le 1$,}
\]
and \eqref{uno1} holds. \\
Finally, similar arguments allow us to prove that
\[
\left\vert\frac{\partial\J}{\partial v}(u_n,v_n)[z] - \frac{\partial\J}{\partial v}(u, v)[z]\right\vert \longrightarrow 0 
\quad \mbox{ uniformly with respect to $z \in X_2$,
$\|z\|_{X_2} \le 1$;}
\]
hence, \eqref{due2} is satisfied, too.
\end{proof}

\begin{remark}
Differently from the Sobolev Embedding Theorem, in our setting we have 
\[
X_i \hookrightarrow L^r(\Omega) \quad \forall r\geq 1\; \hbox{ and }\; i \in \{1, 2\}.
\] 
Hence, in the proof of Proposition \ref{smooth1} no sub--critical growth is required for $G_u(x,u,v)$ 
and $G_v(x,u,v)$.
\end{remark}


\section{The weak Cerami--Palais--Smale condition}
\label{sec_wcps}

In order to prove some more properties of functional $\J :X \to \R$
defined as in \eqref{functional}, 
we require that $R \geq 1$ exists such that the following conditions hold:
\begin{itemize}
\item[$(h_2)$] a constant $\mu_0 > 0$ exists such that
\[
A(x, u)\geq \mu_0 \ \mbox{ and } \ B(x, v)\geq\mu_0 \quad\mbox{ a.e. in } \Omega, \mbox{ for all } u, v\in\R;
\]
\item[$(h_3)$] there exists $\mu_1> 0$ such that
\[
\begin{split}
&A(x, u) +\frac{1}{p_1} A_u(x, u) u\geq \mu_1 A(x, u) \quad \mbox{ a.e. in } \Omega \mbox{ if } \vert u\vert\geq R,\\
&B(x, v) +\frac{1}{p_2} B_v(x, v) v\geq\mu_1 B(x, v) \quad \mbox{ a.e. in } \Omega \mbox{ if } \vert v\vert\geq R;
\end{split}
\]
\item[$(h_4)$] there exist $\theta_1$, $\theta_2$, $\mu_2 >0$ such that
\begin{equation}    \label{thetai}
\theta_1\ < \ \frac{1}{p_1}, \qquad
\theta_2\ < \ \frac{1}{p_2},
\end{equation}
and
\[
\begin{split}
&(1 - p_1  \theta_1) A(x, u) - \theta_1 A_u(x, u)u\ \geq\ \mu_2 A(x, u) \quad \mbox{a.e. in } \Omega, \mbox{ for all } u\in\R,\\
&(1 - p_2  \theta_2) B(x, v) - \theta_2 B_v(x, v)v\ \geq\ \mu_2 B(x, v) \quad \mbox{a.e. in } \Omega, \mbox{ for all } v\in\R;
\end{split}
\]
\item[$(g_2)$] taking $\theta_1$, $\theta_2$ as in $(h_4)$, we have
\[
0 < G(x, u, v) \leq \theta_1 G_u(x, u, v)u + \theta_2 G_v(x, u, v)v 
\quad \mbox{ a.e. in }\Omega \mbox{ if }\vert (u, v)\vert\geq R.
\]
\end{itemize}

\begin{remark}    \label{remmu1}
In $(h_3)$, respectively $(h_4)$, we can always assume that
\[
\mu_1 <1\quad\hbox{and}\quad
\mu_2 <\ \min\ \lbrace 1 - p_1 \theta_1, 1 - p_2 \theta_2\rbrace.
\] 
Hence, we have that
\[
\begin{split}
&-p_1(1-\mu_1) A(x, u)\ \leq\ A_u(x, u) u\ \leq\
 \frac{1}{\theta_1}\ (1 - p_1 \theta_1 -\mu_2)\ A(x, u) \quad\mbox{a.e. in } \Omega \mbox{ if } \vert u\vert\geq R,\\
&-p_2(1-\mu_1)B(x, v)\ \leq\ B_v(x, v) v
\ \leq\ \frac{1}{\theta_2}\ (1 - p_2 \theta_2 -\mu_2)\ B(x, v) \quad\mbox{a.e. in } \Omega \mbox{ if } \vert v\vert\geq R,
\end{split}
\]
which imply
\[
\begin{split}
&\vert A_u(x,u) u\vert\ \leq\ \gamma_1 A(x,u) \quad\mbox{a.e. in } \Omega \mbox{ if } \vert u\vert\geq R,\\
&\vert B_v(x,v) v\vert\ \leq\ \gamma_2 B(x,v) \quad\mbox{a.e. in } \Omega \mbox{ if } \vert v\vert\geq R,
\end{split}
\]
where $\gamma_i = \max\lbrace p_i(1-\mu_1),  \frac{1}{\theta_i}\ (1 - p_i \theta_i -\mu_2)\rbrace$, $i \in \{1,2\}$. 
Then, it results
\begin{equation}  \label{Auremark}
\begin{split}
&\vert A_u(x, u)\vert\leq\frac{\gamma_1}{R} A(x, u) \quad\mbox{a.e. in } \Omega \mbox{ if } \vert u\vert\geq R,\\
&\vert B_v(x, v)\vert\leq\frac{\gamma_2}{R} B(x, v) \quad\mbox{a.e. in } \Omega \mbox{ if } \vert v\vert\geq R. 
\end{split}
\end{equation}
On the other hand, from hypotheses $(h_0)$--$(h_2)$, $(h_4)$ and standard computations it follows that
\begin{equation}    \label{Atheta}
A(x, u)\leq a_1 + a_2 \vert u\vert^{ \frac{1}{\theta_1}\ (1 - p_1 \theta_1 -\mu_2)} 
\quad \mbox{ a.e. in } \Omega, \mbox{ for all } u\in\R,
\end{equation}
\begin{equation}    \label{Btheta}
B(x, v)\leq b_1 + b_2\vert v\vert^{ \frac{1}{\theta_2}\ (1 - p_2 \theta_2 -\mu_2)} 
\quad \mbox{ a.e. in } \Omega, \mbox{ for all } v\in\R,
\end{equation}
for suitable constants $a_1, a_2, b_1, b_2 >0$.
\end{remark}

\begin{example}\label{ex0}
Let $p_i > 1$, with $i \in\{1,2\}$, and consider
\begin{equation}\label{ex01}
A(x,u) \ =\ A_1(x) + A_2(x) |u|^{\gamma_1},\quad 
B(x,v) \ =\ B_1(x) + B_2(x) |v|^{\gamma_2}
\end{equation}
for a.e. $x \in \Omega$, all $u, v \in \R$, where 
\begin{equation}\label{ex05}
\gamma_i >1\qquad  \hbox{for $i \in \{1,2\}$,} 
\end{equation}
and
\begin{equation}\label{ex02}
A_i \in L^{\infty}(\Omega)\quad \hbox{and}\quad 
B_i \in L^{\infty}(\Omega)\qquad \hbox{for $i \in \{1,2\}$.}
\end{equation}
If we assume that a constant $\mu_0 > 0$ exists such that
\begin{equation}
\label{ex03}
A_1(x),\ B_1(x) \ge \mu_0,\qquad A_2(x),\ B_2(x) \ge 0\qquad \hbox{for a.e. $x \in \Omega$,}
\end{equation}
then conditions $(h_0)$--$(h_3)$ are satisfied with 
\[
A_u(x,u)\ = \ \gamma_1A_2(x) |u|^{\gamma_1-2} u,
\quad B_v(x,v) \ =\ \gamma_2 B_2(x) |v|^{\gamma_2 -2} v
\quad \hbox{for a.e. $x \in \Omega$, all $u, v \in \R$.} 
\]
Moreover, if we take
\begin{equation}\label{ex06}
0 < \theta_i < \frac{1}{p_i+\gamma_i}\quad  \hbox{for $i \in \{1,2\}$,} 
\end{equation}
also $(h_4)$ holds with
$\mu_2 = \min\{1-p_1\theta_1-\gamma_1\theta_1,\ 1-p_2\theta_2-\gamma_2\theta_2\} > 0$.
\end{example}

\begin{remark}  \label{rmkPerera0}
From $(g_0)$ and $(g_2)$ it follows that
\[
\begin{split}
&0 < G(x,u,0)\le \theta_1 G_u(x,u,0)u \quad \mbox{ for a.e. } x\in\Omega, \mbox{ if } |u| \ge R,\\
&0 < G(x,0,v)\le \theta_2 G_v(x,0,v)v \quad \mbox{ for a.e. } x\in\Omega, \mbox{ if } |v| \ge R.
\end{split}
\]
Hence, standard computations imply that
$h_i\in L^{\infty}(\Omega)$, $h_i(x) >0$ for a.a. $x\in\Omega$, $i \in \{1,2\}$,  
exist such that
\begin{eqnarray}
\label{rmkPerera}
G(x,u,0) &\geq& h_1(x) |u|^{\frac{1}{\theta_1}}\quad \mbox{ for a.e. } x\in\Omega, \mbox{ if } |u| \ge R,\\
G(x,0,v) &\geq& h_2(x) |v|^{\frac{1}{\theta_2}}\quad \mbox{ for a.e. } x\in\Omega, \mbox{ if } |v| \ge R.\nonumber
\end{eqnarray}
Thus, if also \eqref{gucr} and \eqref{gvcr} hold, from 
\eqref{thetai} and \eqref{Gsigma} it follows that 
\begin{equation}    \label{minmaxPerera}
p_1 <\ \frac{1}{\theta_1}\ \le q_1, \qquad
p_2 <\ \frac{1}{\theta_2}\ \le q_2.
\end{equation}
\end{remark}

\begin{example}\label{ex1}
Consider
\begin{equation}\label{ex11}
G(x,u,v) \ =\ |u|^{q_1} + c_* |u|^{\gamma_3} |v|^{\gamma_4} + |v|^{q_2}
\quad\hbox{for a.e. $x \in \Omega$, all $u, v \in \R$,}
\end{equation}
with $c_* \in \R$ and
\begin{equation}\label{ex13}
q_1 >1,\quad q_2 >1, \quad
\gamma_3 > 1,\quad \gamma_4 > 1.
\end{equation}
Then, condition $(g_0)$ is satisfied with 
\[
\begin{split}
&G_u(x,u,v)\ = \ q_1 |u|^{q_1-2} u + \gamma_3 c_* |u|^{\gamma_3-2} u |v|^{\gamma_4},\\
&G_v(x,u,v)\ = \ \gamma_4 c_* |u|^{\gamma_3} |v|^{\gamma_4-2} v + q_2 |v|^{q_2 -2} v
\end{split} 
\]
for a.e. $x \in \Omega$, all $u, v \in \R$.
Moreover, if also
\begin{equation}\label{ex17}
\gamma_3 < q_1\quad \hbox{and}\quad \gamma_4 < q_2, 
\end{equation}
then from Young inequality and direct computations it follows that
\eqref{gucr} and \eqref{gvcr} hold with
\begin{equation}\label{ex16}
\begin{split}
&s_1 = \gamma_4 \alpha_1,\quad \hbox{with}\;
\alpha_1 = \frac{q_1-1}{q_1 -\gamma_3} > 1\ \iff\  q_1 -1 = (\gamma_3 -1) \frac{\alpha_1}{\alpha_1 -1},\\
&s_2 = \gamma_3 \alpha_2, \quad \hbox{with}\; \alpha_2 = \frac{q_2-1}{q_2 -\gamma_4} > 1
\  \iff \  q_2 - 1 = (\gamma_4-1) \frac{\alpha_2}{\alpha_2 -1},
\end{split}
\end{equation}
for a suitable $\sigma > 0$.
At last, if 
\begin{equation}\label{ex12}
c_* \ge 0\qquad \hbox{and}\qquad
\frac{\gamma_3}{q_1} + \frac{\gamma_4}{q_2} \ge 1, 
\end{equation}
taking 
\begin{equation}\label{ex15}
 \theta_i\ \ge\ \frac{1}{q_i}\quad  \hbox{for $i \in \{1,2\}$,} 
\end{equation}
we have that also $(g_2)$ is verified.\\
We note that if $G(x,u,v)$ in \eqref{ex11}
has $c_* = 0$, then problem \eqref{euler} with such a nonlinear term
reduces to an \textsl{uncoupled system}. 
\end{example}

\begin{example}\label{ex2}
Consider
\begin{equation}\label{ex21}
G(x,u,v) \ =\ |u|^{q_1} + |u|^{\gamma_3} \lg(v^2+1) + \lg(u^2+1) |v|^{\gamma_4} + |v|^{q_2}
\quad\hbox{for a.e. $x \in \Omega$, all $u, v \in \R$,}
\end{equation}
with 
\begin{equation}\label{ex26}
q_1 >1,\quad q_2 >1, \quad
\gamma_3 > 1,\quad \gamma_4 > 1.
\end{equation}
Then, condition $(g_0)$ is satisfied with 
\[
\begin{split}
&G_u(x,u,v)\ = \ q_1 |u|^{q_1-2} u + \gamma_3 |u|^{\gamma_3-2} u \lg(v^2+1) +
\frac{2u}{u^2+1} |v|^{\gamma_4},\\
&G_v(x,u,v)\ = \ |u|^{\gamma_3} \frac{2v}{v^2+1} +
\gamma_4 \lg(1+u^2) |v|^{\gamma_4-2} v + q_2 |v|^{q_2 -2} v
\end{split} 
\]
for a.e. $x \in \Omega$, all $u, v \in \R$.
Moreover, assume also that
\begin{equation}\label{ex27}
\gamma_3 < q_1\quad \hbox{and}\quad \gamma_4 < q_2. 
\end{equation}
From \eqref{ex26} and \eqref{ex27}, taking any $\eps > 0$ such that
\begin{equation}\label{ex28}
\eps < \gamma_4 \frac{q_1 - \gamma_3}{q_1 - 1}\qquad \then\qquad
\eps < \gamma_4 \quad\hbox{and}\quad 0 < \frac{\gamma_3 \gamma_4 -\eps}{\gamma_4-\eps} \le q_1
\end{equation}
direct computations imply that a suitable constant $C > 0$ exists such that
\[
0 \le \lg(v^2+1) \le |v|^{\eps} + C\quad \hbox{for all $v \in\R$,}
\]
then from Young inequality applied to $\frac{\gamma_4}{\eps} > 1$
and \eqref{ex27}, \eqref{ex28} we have that
\[
\begin{split}
||u|^{\gamma_3-2} u \lg(v^2+1)| &\le |u|^{\gamma_3-1} |v|^{\eps} + C |u|^{\gamma_3-1}\\ 
&\le C (|u|^{\frac{\gamma_3 \gamma_4 -\eps}{\gamma_4-\eps} - 1} + |v|^{\gamma_4} + |u|^{\gamma_3-1}) 
\le C (1 + |u|^{q_1 - 1} + |v|^{\gamma_4}).
\end{split}
\]
Thus, by applying similar arguments also in order to estimate 
$\lg(1+u^2) |v|^{\gamma_4-2} v$ we prove that 
\eqref{gucr} and \eqref{gvcr} hold with 
\begin{equation}\label{ex23}
s_1 = \gamma_4\quad  \hbox{and}\quad s_2 = \gamma_3,
\end{equation}
for a suitable $\sigma > 0$.
At last, taking 
\begin{equation}\label{ex25}
 \theta_1\ \ge\ \frac{1}{\gamma_3}\quad  \hbox{and}\quad 
 \theta_2\ \ge\ \frac{1}{\gamma_4}
\end{equation}
we have that also $(g_2)$ is verified
as \eqref{ex27} implies $\theta_i q_i \ge 1$ for $i\in \{1,2\}$.
\end{example}

In order to prove the weak Cerami--Palais--Smale condition, we 
require the following boundedness result (for its proof, see \cite[Theorem II.5.1]{LU}).

\begin{lemma}    \label{Ladyz}
Let $\Omega$ be an open bounded subset of $\R^N$ and consider $\xi\in W_0^{1, p}(\Omega)$ with $p \le N$. 
Suppose that $\gamma > 0$ and $k_0\in\N$ exist such that 
\[
\int_{\Omega_{k}^{+}} \vert\nabla \xi\vert^p dx\ \leq\ 
\gamma\left( \int_{\Omega^{+}_{k}} (\xi - k)^r dx\right)^{\frac{p}{r}} 
+ \gamma \sum_{j=1}^m k^{\alpha_j} [\meas(\Omega^{+}_{k})]^{1-\frac{p}{N}+\varepsilon_j}\quad
\hbox{for all $k\geq k_0$,}
\]
with $\Omega^{+}_{k} :=\lbrace x\in\Omega : \xi(x) > k\rbrace$ and $r, m, \alpha_j, \varepsilon_j$ 
are positive constants such that
\[
1\leq r < p^{\ast}, \qquad \varepsilon_j > 0, \qquad p\leq\alpha_j < \varepsilon_j p^{\ast} + p.
\]
Then $\displaystyle\esssup_{\Omega} \xi$ is bounded from above by a positive constant 
which can be chosen so that it depends only on 
$\meas(\Omega), N, p, \gamma, k_0, r, m, \varepsilon_j, \alpha_j, \vert \xi\vert_{p^{\ast}}$ (eventually,
$\vert \xi\vert_{l}$ for some $l > r$ if $p^{\ast} = +\infty$).
\end{lemma}

Up to now, no upper bound is required for the growth of the nonlinear term
$G(x,u,v)$. Anyway, subcritical assumptions need for proving the 
weak Cerami--Palais--Smale condition at any level, more precisely
we will require that the whole assumption $(g_1)$ is satisfied,
i.e., the exponents in \eqref{gucr} and \eqref{gvcr} are such that
\eqref{crit_exp} and \eqref{crit_expi} hold.

\begin{remark} \label{rmkcrit}
Consider $1 < p_1 < N$ and $1 < p_2 < N$. 
If $s_1 \in \R$ is as in \eqref{crit_expi},
then $s_3 > 0$ exists such that
\begin{equation} \label{crits11}
1< s_3 < p_1^*, \qquad 
0 \le s_4:=\ \frac{s_1 s_3}{s_3 - 1} < \ \frac{p_1}{N}\ p_2^* < p_2^*.
\end{equation}
In fact, \eqref{crit_expi} implies that
\[
p_1 p_2^* - N s_1 > 0\qquad \hbox{and}\qquad
1\ <\ \frac{p_1 p_2^*}{p_1 p_2^* - N s_1}\ <\ p^*_1,
\]
so $s_3 \in ]\frac{p_1 p_2^*}{p_1 p_2^* - N s_1}, p^*_1[$ exists such that
\eqref{crits11} holds.\\
Thus, from \eqref{crits11} and Young inequality it follows that
\begin{equation} \label{crits12}
|u|\ |v|^{s_1} \ \le \ \frac{|u|^{s_3}}{s_3}\ + \
\frac{s_3-1}{s_3}\ |v|^{s_4}\qquad \hbox{for all $(u,v) \in \R^2$}. 
\end{equation}
Similarly, if $s_2 \in \R$ is as in \eqref{crit_expi},
then $s_5 > 0$ exists such that
\begin{equation} \label{crits21}
1< s_5 < p_2^*, \qquad 
0 \le s_6:=\ \frac{s_2 s_5}{s_5 - 1} 
<\ \frac{p_2}{N} \ p_1^* < p_1^*,
\end{equation}
and
\[
|u|^{s_2}\ |v| \ \le \ \frac{s_5-1}{s_5}\ |u|^{s_6} +
\frac{|v|^{s_5}}{s_5} \qquad \hbox{for all $(u,v) \in \R^2$}. 
\]
\end{remark}

\begin{remark}   \label{rmkcrit1}
Taking $1 < p_1 < N$ and $1 < p_2 < N$, if $(g_0)$--$(g_1)$ hold, 
then from \eqref{Gsigma} and the notation and estimates in Remark \ref{rmkcrit} 
it follows that 
\[
\vert G(x,u,v)\vert\leq C(1+\vert u\vert +\vert u\vert^{q_1} +\vert u\vert^{s_3}  
+\vert u\vert^{s_6} +\vert v\vert +\vert v\vert^{q_2} +\vert v\vert^{s_4} +\vert v\vert^{s_5})
\]
a.e. in $\Omega$, for all $(u,v) \in \R^2$.
Hence, putting 
\begin{equation}    \label{Ggrow}
\overline{q}_1 =\max\lbrace q_1, s_3, s_6\rbrace\qquad
\hbox{and}\qquad
\overline{q}_2 = \max\lbrace q_2, s_4, s_5\rbrace, 
\end{equation}
up to change the constant $C > 0$ we have that
\begin{equation}    \label{Gsigmamax}
\vert G(x,u,v)\vert\ \leq\ C (1 + \vert u\vert +\vert v\vert +\vert u\vert^{\overline{q}_1} +\vert v\vert^{\overline{q}_2})
\quad \hbox{for a.e. $x \in \Omega$, all $(u,v) \in \R^2$}.
\end{equation}
At last, if also $(g_2)$ holds,
then from \eqref{minmaxPerera}, \eqref{crits11} and \eqref{crits21} 
it follows that
\begin{equation}    \label{minmaxPerera1}
p_1 <\ \frac{1}{\theta_1}\ \le \overline{q}_1 < p_1^{\ast}, \qquad
p_2 <\ \frac{1}{\theta_2}\ \le \overline{q}_2 < p_2^{\ast}.
\end{equation}
\end{remark}

\begin{proposition}     \label{PropwCPS}
Assume hypotheses $(h_0)$--$(h_4)$, $(g_0)$--$(g_2)$ hold. 
Then, functional $\J$ in \eqref{functional} satisfies condition $(wCPS)$ in $\R$.
\end{proposition}

\begin{proof}
Let $\beta\in\R$ be fixed. We have to prove that if $((u_n, v_n))_n\subset X$ is a sequence such that
\begin{equation}  \label{3.15}
\J(u_n, v_n)\longrightarrow \beta \quad \mbox{ and }\quad 
\Vert d\J(u_n, v_n)\Vert_{X^{\prime}}(1 + \Vert (u_n, v_n)\Vert_{X})\longrightarrow 0 \quad \mbox{ if } n\to +\infty,
\end{equation}
then $(u, v)\in X$ exists such that 
\begin{enumerate}
\item[$(i)$] $(u_n,v_n)\longrightarrow (u, v)$ in $W$ (up to subsequences),
\item[$(ii)$] $\J(u, v) = \beta$, $\  d\J(u, v) = 0$.
\end{enumerate}
To this aim, our proof is divided in several steps; more precisely, we will prove that:
\begin{enumerate}
\item $((u_n, v_n))_n$ is bounded in $W$; hence, there exists $(u,v)\in W$ such that,
up to subsequences, we have that 
\begin{align}         \label{3.16}
&(u_n, v_n)\rightharpoonup (u, v) \mbox{ weakly in } W,\\        \label{3.17}
&(u_n, v_n)\longrightarrow (u, v) \mbox{ in } L^{r_1}(\Omega)\times L^{r_2}(\Omega) 
\quad \mbox{ for any } (r_1,r_2) \in [1, p_1^{\ast}[ \times [1, p_2^{\ast}[,\\ 
  \label{3.18}
&(u_n, v_n)\longrightarrow (u, v) \mbox{ a.e. in } \Omega;
\end{align}
\item $(u, v)\in L$;
\item for any $k>0$, defining $T_k:\R\rightarrow\R$ such that
\[
T_k t:= \begin{cases}
t &\hbox{ if } \vert t\vert\leq k\\
k \frac{t}{\vert t\vert} &\hbox{ if } \vert t\vert > k
\end{cases}
\]
and 
\[
\T_k:(t_1, t_2)\in\R^2\mapsto\T_k(t_1, t_2) = (T_k t_1, T_k t_2)\in\R^2,
\]
then, if $k\geq \max\lbrace\vert (u, v)\vert_{L}, R\rbrace + 1$ (with $R\geq 1$ as in our set of hypotheses), 
we have
\begin{align}      \label{step3.1}
\Vert d\J (\T_k (u_n, v_n))\Vert_{X^{\prime}}\longrightarrow 0
\end{align}
and
\begin{align}      \label{step3.2}
\J (\T_k(u_n, v_n))\longrightarrow \beta;
\end{align}
\item $\Vert(T_k u_n -u, T_k v_n -v)\Vert_{W}\longrightarrow 0$ and then $(i)$ holds;
\item $(ii)$ is satisfied.
\end{enumerate}
For simplicity, here and in the following we use the notation $(\varepsilon_n)_n$ for any infinitesimal sequence 
depending only on $((u_n, v_n))_n$, 
while $(\varepsilon_{k, n})_n$ for any infinitesimal sequence depending also on some fixed integer $k$.
Moreover, we assume that $1 < p_i < N$ for both $i \in \{1,2\}$,
otherwise the proof is simpler.\\
{\sl Step 1.} 
Firstly, we note that \eqref{dJzstar} and \eqref{3.15} imply that
\[
\frac{\partial \J}{\partial u}(u_n,v_n)[u_n] = \eps_n, \quad
\frac{\partial \J}{\partial v}(u_n,v_n)[v_n] = \eps_n.
\]
Then, taking $\theta_1$, $\theta_2$ as in $(h_4)$ and $(g_2)$, 
and fixing any $n \in \N$, from
\eqref{functional}, \eqref{dJw}, \eqref{dJz}, \eqref{3.15} and conditions $(h_2)$, $(h_4)$,
it follows that
\[
\begin{split}
\beta + \varepsilon_n\ =\ & \J(u_n, v_n) - \theta_1 \frac{\partial \J}{\partial u}(u_n,v_n)[u_n] 
- \theta_2 \frac{\partial \J}{\partial v}(u_n,v_n)[v_n]\\
=\ & \frac{1}{p_1}\int_{\Omega}[(1 - \theta_1 p_1) A(x,u_n) - \theta_1 A_u(x,u_n) u_n] \vert\nabla u_n\vert^{p_1} dx\\
&\quad + \frac{1}{p_2}\int_{\Omega}[(1 - \theta_2 p_2) B(x,v_n) - \theta_2 B_v(x,v_n) v_n] \vert\nabla v_n\vert^{p_2} dx\\
&\quad +\int_{\Omega} (\theta_1 G_u(x,u_n,v_n)u_n + \theta_2 G_v(x,u_n,v_n)v_n - G(x,u_n,v_n)) dx\\
\ge\ & \frac{\mu_0 \mu_2}{p_1}\ \|u_n\|_{W_1}^{p_1} + \frac{\mu_0 \mu_2}{p_2}\ \|v_n\|_{W_2}^{p_2}
\ -\ C
\end{split}
\]
as \eqref{gucr}, \eqref{gvcr}, $(g_2)$ and \eqref{Gsigma} imply
\[
\int_{\Omega} (\theta_1 G_u(x,u_n,v_n)u_n + 
\theta_2 G_v(x,u_n,v_n)v_n - G(x,u_n,v_n)) dx\ \ge\ - C.
\]
Hence, $(u,v)\in W$ exist such that \eqref{3.16}--\eqref{3.18} hold, up to subsequences.\\
{\sl Step 2.} As we want to prove that $(u, v)\in L$, arguing by contradiction we assume that 
either $u\notin L^{\infty}(\Omega)$ or $v\notin L^{\infty}(\Omega)$.
If $u\notin L^{\infty}(\Omega)$, either
\begin{equation}    \label{sup_u}
\esssup_{\Omega} u = +\infty
\end{equation}
or
\begin{equation}     \label{sup_menou}
\esssup_{\Omega}(-u) = +\infty. 
\end{equation}
For example, suppose that \eqref{sup_u} holds.
Then, for any fixed $k\in\N$, $k > R$ (with $R\geq 1$ as in the hypotheses), we have
\begin{align}         \label{mpos}
\meas(\Omega_{k}^{+}) > 0,
\end{align}
with $\Omega_{k}^{+}:= \left\lbrace x\in\Omega \mid u(x) > k\right\rbrace$.
Now, we consider the new function $R^{+}_k: t\in\R \mapsto R^{+}_k t\in\R$ such that
\[
R^{+}_k t = \begin{cases}
0 &\hbox{ if } t\leq k\\
t - k &\hbox{ if } t > k
\end{cases}.
\]
From \eqref{3.16}, we have
\[
R^{+}_k u_n \rightharpoonup R^{+}_k u \quad \mbox{ in } W_1.
\]
Thus, by the sequentially weakly lower semicontinuity of $\Vert\cdot\Vert_{W_1}$, it follows that
\[
\int_{\Omega} \vert\nabla R^{+}_k u\vert^{p_1} dx 
\leq \liminf_{n\to +\infty} \int_{\Omega} \vert\nabla R^{+}_k u_n\vert^{p_1} dx,
\]
i.e.,
\begin{equation}        \label{3.23}
\int_{\Omega_{k}^{+}} \vert\nabla u\vert^{p_1} dx
 \leq \liminf_{n\to +\infty}\int_{\Omega_{n, k}^{+}} \vert\nabla u_n\vert^{p_1} dx,
\end{equation}
with $\Omega_{n, k}^{+}:=\left\lbrace x\in\Omega\mid u_n(x) > k\right\rbrace$.\\
On the other hand, by definition, we have
\[
\Vert R^{+}_k u_n\Vert_{X_1} \leq \Vert u_n\Vert_{X_1}.
\]
Hence, \eqref{dJzstar} and \eqref{3.15} imply that
\[
\frac{\partial\J}{\partial u}(u_n, v_n)[R^{+}_k u_n]\ \longrightarrow\ 0;
\]
thus, from \eqref{mpos} an integer $n_k\in\N$ exists such that
\begin{equation}      \label{m+mpos}
\frac{\partial\J}{\partial u}(u_n, v_n)\left[ R^{+}_k u_n\right]
 < \meas(\Omega^{+}_{k}) \quad \mbox{ for all } n\geq n_k.
\end{equation}
Taking any $k > R$ and $n \in \N$, from \eqref{dJw}, hypothesis 
$(h_3)$ with $\mu_1 <1$ (see Remark \ref{remmu1})
and condition $(h_2)$ it follows that 
\begin{align*}
\frac{\partial\J}{\partial u}(u_n, v_n)\left[R^{+}_k u_n\right] 
=\ &\int_{\Omega^{+}_{n,k}}\left( 1- \frac{k}{u_n}\right) \left(A(x, u_n) + 
\frac{1}{p_1}A_u(x, u_n)u_n\right) \vert\nabla u_n\vert^{p_1} dx\\
&+\int_{\Omega^{+}_{n,k}} \frac{k}{u_n} A(x, u_n) \vert\nabla u_n\vert^{p_1} dx 
- \int_{\Omega} G_u(x,u_n,v_n) R^{+}_k u_n dx\\
\geq\ & \mu_1 \int_{\Omega^{+}_{n,k}} A(x, u_n)\vert\nabla u_n\vert^{p_1} dx 
- \int_{\Omega} G_u(x,u_n,v_n) R^{+}_k u_n dx\\
\geq\ & \mu_0\mu_1\int_{\Omega^{+}_{n,k}} \vert\nabla u_n \vert^{p_1} dx 
- \int_{\Omega} G_u(x,u_n,v_n) R^{+}_k u_n dx,
\end{align*}
which, together with \eqref{m+mpos}, implies that
\begin{equation}    \label{mu1lambda}
\mu_0\mu_1 \int_{\Omega^{+}_{n,k}} \vert\nabla u_n\vert^{p_1} dx 
\leq \meas(\Omega^{+}_{k}) + \int_{\Omega} G_u(x,u_n,v_n) R^{+}_k u_n dx
\quad \mbox{ for all } n\geq n_k.
\end{equation}
We note that, from \eqref{3.18} and $(g_0)$, we have
\[
G_u(x, u_n, v_n) R^{+}_k u_n\longrightarrow G_u(x, u, v)R^{+}_k u \quad\mbox{ a.e. in } \Omega.
\]
Moreover, hypothesis \eqref{crit_expi} implies
that suitable exponents can be choosen so that 
\eqref{crits11} holds.
Thus, since $|R^{+}_k u_n(x)| \le |u_n(x)|$ for a.e. $x \in \Omega$,
by using \eqref{crits12} in \eqref{gucr}, from \eqref{crit_exp},
\eqref{crits11} and \eqref{3.17} 
we have that a function $h\in L^1(\Omega)$ exists such that
\[
\vert G_u(x,u_n,v_n) R^{+}_k u_n\vert
\leq \ C (\vert u_n\vert +\vert u_n\vert^{q_1} +\vert u_n\vert^{s_3}
+\vert v_n\vert^{s_4}) \leq h(x) 
\quad \hbox{for a.e. $x \in \Omega$,}
\]
up to subsequences (see, e.g., \cite[Theorem 4.9]{Br}). 
So, by using the Dominated Convergence Theorem, we have
\begin{equation}     \label{limGuRk}
\lim_{n\to +\infty} \int_{\Omega} G_u(x,u_n,v_n) R^{+}_k u_n dx 
= \int_{\Omega} G_u(x,u,v) R^{+}_k u dx.
\end{equation}
Hence, summing up, from \eqref{3.23}, \eqref{mu1lambda} and \eqref{limGuRk} we obtain
\[
\mu_0 \mu_1 \int_{\Omega^{+}_k}\vert\nabla u\vert^{p_1} dx 
\leq \meas(\Omega^{+}_k) +\int_{\Omega^{+}_k} G_u(x, u, v) R^{+}_k u dx,
\]
which implies, by using again \eqref{gucr}, that
\begin{equation}    \label{gradu1}
\int_{\Omega^{+}_k} \vert\nabla u\vert^{p_1} dx\leq 
C\left(\meas(\Omega^{+}_k) + \int_{\Omega^{+}_k} \vert u\vert dx + 
\int_{\Omega^{+}_k}\vert u\vert^{q_1} dx 
+ \int_{\Omega^{+}_k}\vert u\vert \vert v\vert^{s_1} dx\right).
\end{equation}
Since \eqref{crit_expi} holds, taking $s_3$, $s_4$ as in \eqref{crits11}, 
from \eqref{crits12}, H\"older inequality with conjugate exponents
$\frac{p^*_2}{s_4} > 1$ (see \eqref{crits11}) 
and $\frac{p^*_2}{p^*_2 - s_4}$, and direct computations
it follows that
\begin{equation}    \label{gradu2}
\begin{split}
\int_{\Omega^{+}_k} \vert u\vert \vert v\vert^{s_1} dx
&\leq \ \frac{1}{s_3} \int_{\Omega^{+}_k}\vert u\vert^{s_3} dx +\
\frac{s_3-1}{s_3}\ \int_{\Omega^{+}_k}\vert v\vert^{s_4} dx\\
&\leq \ \frac{1}{s_3} \int_{\Omega^{+}_k}\vert u\vert^{s_3} dx +\
\frac{s_3-1}{s_3} \ \vert v\vert_{p_2^*}^{s_4}\ [\meas(\Omega^{+}_k)]^{1-\frac{s_4}{p_2^*}}.
\end{split}
\end{equation}
Taking $\overline{q}_1$ as in \eqref{Ggrow}, in our setting of hypotheses
the estimates in \eqref{minmaxPerera1} hold, so 
\[
u \in W_1 \hookrightarrow L^{\overline{q}_1}(\Omega),
\]
and, since $k \ge 1$, we have that
\[
\int_{\Omega^{+}_k} \vert u\vert dx \le \int_{\Omega^{+}_k}\vert u\vert^{\overline{q}_1} dx,
\quad \int_{\Omega^{+}_k} \vert u\vert^{q_1} dx \le \int_{\Omega^{+}_k}\vert u\vert^{\overline{q}_1} dx,
\quad \int_{\Omega^{+}_k} \vert u\vert^{s_3} dx \le \int_{\Omega^{+}_k}\vert u\vert^{\overline{q}_1} dx.
\]
Thus, \eqref{gradu1} and \eqref{gradu2} imply that
\begin{equation}    \label{gradu3}
\int_{\Omega^{+}_k} \vert\nabla u\vert^{p_1} dx\leq 
C\left(\int_{\Omega^{+}_k}\vert u\vert^{\overline{q}_1} dx 
+ \meas(\Omega^{+}_k) + [\meas(\Omega^{+}_k)]^{1-\frac{s_4}{p_2^*}}\right)
\end{equation}
with $C = C(\|v\|_{W_2}) > 0$.
At last, from \eqref{minmaxPerera1} and direct computations we obtain that
\[
\begin{split}
\int_{\Omega^{+}_k} \vert u\vert^{\overline{q}_1} dx 
&\leq\ 2^{\overline{q}_1 -1}\int_{\Omega^{+}_k} \vert u - k\vert^{\overline{q}_1} dx 
+ 2^{\overline{q}_1 -1} k^{\overline{q}_1} \meas(\Omega^{+}_k)\\
&\leq\ 2^{\overline{q}_1 -1} |u|_{\overline{q}_1}^{\overline{q}_1 - p_1}
\left(\int_{\Omega^{+}_k} \vert u - k\vert^{\overline{q}_1} dx\right)^{\frac{p_1}{\overline{q}_1}} 
+ 2^{\overline{q}_1 -1} k^{\overline{q}_1} \meas(\Omega^{+}_k);
\end{split}
\]
while, as $p_1 < N$, from \eqref{crits11} it results
\[
\begin{split}
&\meas(\Omega^{+}_k) = \meas(\Omega^{+}_k)^{1-\frac{p_1}{N} + \eps_1}, \quad \eps_1 = \frac{p_1}{N} > 0,\\
&\meas(\Omega^{+}_k)^{1-\frac{s_4}{p_2^*}} = 
\meas(\Omega^{+}_k)^{1-\frac{p_1}{N} + \eps_2}, \quad \eps_2 = \frac{p_1}{N} - \frac{s_4}{p_2^*} > 0.
\end{split}
\]
Thus, summing up, from \eqref{gradu3} we have that
\[
\int_{\Omega^{+}_k} \vert\nabla u\vert^{p_1} dx
\leq 
C\left(\left(\int_{\Omega^{+}_k} \vert u - k\vert^{\overline{q}_1} dx\right)^{\frac{p_1}{\overline{q}_1}} 
+ \meas(\Omega^{+}_k)^{1-\frac{p_1}{N} + \eps_1} + \meas(\Omega^{+}_k)^{1-\frac{p_1}{N} + \eps_2}\right),
\]
with $C = C(\|u\|_{W_1},\|v\|_{W_2}) > 0$;
so, as $k \ge 1$ implies $1 \le k^{p_1}$, Lemma \ref{Ladyz} applies and yields
a contradiction to \eqref{sup_u}. \\
Now, suppose that \eqref{sup_menou} holds which implies that,
fixing any $k\in\N$, $k\geq R$, it is
\[
\meas(\Omega^{-}_k)>0 \quad \hbox{with}\; \Omega^{-}_k = \lbrace x\in\Omega: u(x) < -k\rbrace.
\]
In this case, by replacing function $R^{+}_k$ with $R^{-}_k:t\in\R\mapsto R^{-}_k t\in\R$ such that
\[
R^{-}_k t:=
\begin{cases}
0 &\hbox{ if } t\geq -k\\
t+k &\hbox{ if } t<-k
\end{cases},
\]
we can reason as above so to apply again Lemma \ref{Ladyz} which yields a contradiction 
to \eqref{sup_menou}. Then, it has to be $u\in L^{\infty}(\Omega)$.\\
Similar arguments but considering $\frac{\partial \J}{\partial v}(u_n,v_n)$ and 
$R^{+}_k v_n$, respectively $R^{-}_k v_n$, and the related
sets, allow us to prove that it has to be also $v\in L^{\infty}(\Omega)$. \\
{\sl Step 3.}
Fixing $k$ as required in this step, define the functions 
\[
R_k: t \in\R\ \mapsto\ R_k t= t - T_k t = \begin{cases}
0 &\hbox{ if } |t|\leq k\\
t - k\frac{t}{|t|} &\hbox{ if } |t|>k
\end{cases} \in \R,
\]
\[
\er_k:(t_1, t_2)\in\R^2\mapsto\er_k(t_1, t_2) = (R_k t_1, R_k t_2)\in\R^2,
\]
and the sets
\[
\Omega^{u}_{n, k}:=\lbrace x\in\Omega: \vert u_n(x)\vert > k\rbrace, 
\quad \Omega^{v}_{n, k}:=\lbrace x\in\Omega: \vert v_n(x)\vert >k\rbrace\qquad
\hbox{for any $n \in \N$.}
\]
By definition, we have that
\begin{equation}    \label{cappa}
\Vert \T_k(u_n,v_n)\Vert_{X}\leq \Vert(u_n, v_n)\Vert_{X} 
\quad \mbox{ and }\quad 
\Vert \er_k(u_n,v_n)\Vert_{X}\leq\Vert (u_n, v_n)\Vert_{X}
\qquad \hbox{for all $n \in \N$;}
\end{equation}
furthermore, we note that
\[
\T_k(u,v) = (u, v) \quad\hbox{and}\quad
\er_k(u,v) = (0,0) \qquad \hbox{for a.e. $x \in \Omega$.}
\]
Then, from \eqref{3.16}--\eqref{3.18} it follows that
\begin{align}   \nonumber
&\T_k(u_n,v_n)\rightharpoonup (u,v) \ \mbox{ weakly in } W,\\ 
\nonumber
&\T_k(u_n,v_n)\longrightarrow (u, v) \;
\mbox{ in } L^{r_1}(\Omega)\times L^{r_2}(\Omega) 
\quad \mbox{ for any } (r_1,r_2) \in [1, p_1^{\ast}[ \times [1, p_2^{\ast}[,\\  
\label{Tkae}
&\T_k(u_n,v_n) \longrightarrow (u, v) \ \mbox{ a.e. in } \Omega,
\end{align}
and also, again from \eqref{3.17} and \eqref{3.18}, we have that
\begin{equation}    \label{Rkstrongly}
\er_k(u_n,v_n)\longrightarrow (0, 0) \;
\mbox{ in } L^{r_1}(\Omega)\times L^{r_2}(\Omega) 
\quad \mbox{ for any } (r_1,r_2) \in [1, p_1^{\ast}[ \times [1, p_2^{\ast}[,
\end{equation}
\[
\meas(\Omega^{u}_{n, k})\longrightarrow 0
\quad\hbox{and}\quad \meas(\Omega^{v}_{n, k})\longrightarrow 0 \quad \mbox{ as } n\to +\infty.
\]
Furthermore, \eqref{dJzstar}, \eqref{3.15} and \eqref{cappa} imply that
\begin{equation}    \label{normRk0}
\left\Vert\frac{\partial\J}{\partial u}(u_n, v_n)\right\Vert_{X^{\prime}_1}\Vert R_k u_n\Vert_{X_1}\longrightarrow 0 
\quad \mbox{ and} \quad 
\left\Vert\frac{\partial\J}{\partial v}(u_n, v_n)\right\Vert_{X^{\prime}_2}\Vert R_k v_n\Vert_{X_2}\longrightarrow 0.
\end{equation}
Now, by reasonig as in the proof of \textsl{Step 2}
but replacing $R_k^{+} u_n$ and $R^{+}_k v_n$ with $R_k u_n$, respectively $R_k v_n$,
from \eqref{normRk0} we have that
\begin{equation} 
 \label{come2u}
\begin{split}
\eps_n &=\ \frac{\partial\J}{\partial u}(u_n, v_n) [R_{k} u_n] + \int_{\Omega}G_u(x, u_n, v_n) R_k u_n dx\\
&\geq \mu_1\int_{\Omega^{u}_{n, k}}A(x, u_n)\vert\nabla u_n\vert^{p_1} dx
\geq \mu_0 \mu_1\int_{\Omega^{u}_{n, k}}\vert\nabla u_n\vert^{p_1}dx, 
\end{split}
\end{equation}
\begin{equation}
\label{come2v}
\begin{split}
\eps_n &=\ \frac{\partial\J}{\partial v}(u_n, v_n)[R_{k} v_n] + \int_{\Omega}G_v(x, u_n, v_n) R_k v_n dx\\
&\geq\mu_1\int_{\Omega^{v}_{n, k}}B(x, v_n)\vert\nabla v_n\vert^{p_2} dx 
\geq \mu_0 \mu_1\int_{\Omega^{v}_{n, k}}\vert\nabla v_n\vert^{p_2}dx,
\end{split}
\end{equation}
as the same arguments used for proving \eqref{limGuRk} apply so that 
from \eqref{Rkstrongly} we obtain
\[
\int_{\Omega}G_u(x, u_n, v_n)R_k u_n dx \longrightarrow 0 \quad \mbox{ and }
\quad\int_{\Omega}G_v(x, u_n, v_n)R_k v_n dx \longrightarrow 0.
\]
Whence, \eqref{come2u} and \eqref{come2v} imply not only that
\[
\int_{\Omega^{u}_{n, k}} \vert\nabla u_n\vert^{p_1} dx\longrightarrow 0 \quad \mbox{ and } \quad
\int_{\Omega^{v}_{n, k}} \vert\nabla v_n\vert^{p_2} dx\longrightarrow 0,
\]
i.e.
\begin{equation}     \label{normRk}
\Vert R_k u_n\Vert_{W_1}\longrightarrow 0 \quad \mbox{ and } \quad
 \Vert R_k v_n\Vert_{W_2}\longrightarrow 0,
\end{equation}
but also
\begin{equation}    \label{Agradto0}
\int_{\Omega^{u}_{n, k}}A(x, u_n) \vert\nabla u_n\vert^{p_1} dx\longrightarrow 0
\quad \mbox{ and } \quad 
\int_{\Omega^{v}_{n, k}}B(x, v_n) \vert\nabla v_n\vert^{p_2} dx\longrightarrow 0.
\end{equation}
Now, in order to prove \eqref{step3.1}, from \eqref{dJbis} it is enough to verify that
\begin{equation}    \label{step33}
\left\|\frac{\partial\J}{\partial u}(\T_k(u_n,v_n))\right\|_{X'_1} \longrightarrow 0
\quad \mbox{ and } \quad 
\left\|\frac{\partial\J}{\partial v}(\T_k(u_n,v_n))\right\|_{X'_2} \longrightarrow 0.
\end{equation}
To this aim, let $w\in X_1$, $z\in X_2$ be such that 
$\Vert w\Vert_{X_1} = 1$, $\Vert z\Vert_{X_2} = 1$.
Then, direct computations imply that
\begin{equation}   \label{sumup}
\begin{split}
\frac{\partial\J}{\partial u}(\T_k(u_n,v_n))[w] 
=\ & \frac{\partial\J}{\partial u}(u_n, v_n)[w] 
-\int_{\Omega_{n, k}^{u}} A(x, u_n)\vert\nabla u_n\vert^{p_1-2}\nabla u_n\cdot\nabla w dx \\
&-\frac{1}{p_1}\int_{\Omega_{n, k}^{u}} A_u(x, u_n) w \vert\nabla u_n\vert^{p_1} dx\\
&+ \int_{\Omega}(G_u(x, u_n, v_n) - G_u(x, T_k u_n, T_k v_n)) w dx
\end{split}
\end{equation}
and 
\begin{equation}   \label{sumup2}
\begin{split}
\frac{\partial\J}{\partial v}(\T_k(u_n,v_n))[z] 
=\ & \frac{\partial\J}{\partial v}(u_n, v_n)[z] 
-\int_{\Omega_{n, k}^{v}} B(x,v_n)\vert\nabla v_n\vert^{p_2-2}\nabla v_n\cdot\nabla z dx \\
&-\frac{1}{p_2}\int_{\Omega_{n,k}^{v}} B_v(x,v_n) z \vert\nabla v_n\vert^{p_2} dx\\
&+ \int_{\Omega}(G_v(x, u_n, v_n) - G_v(x, T_k u_n, T_k v_n)) z dx.
\end{split}
\end{equation}
From \eqref{dJzstar} and \eqref{3.15} we have that
\[
\left|\frac{\partial\J}{\partial u}(u_n, v_n)[w]\right| \le 
\left\|\frac{\partial\J}{\partial u}(u_n, v_n)\right\|_{X_1'} \longrightarrow 0,
\quad
\left|\frac{\partial\J}{\partial v}(u_n, v_n)[z]\right| \le 
\left\|\frac{\partial\J}{\partial v}(u_n, v_n)\right\|_{X_2'} \longrightarrow 0.
\]
Moreover, \eqref{Auremark} and \eqref{Agradto0} imply that
\[
\begin{split}
&\left\vert\int_{\Omega^{u}_{n,k}} A_u(x,u_n) w\vert\nabla u_n\vert^{p_1}dx\right\vert
\leq \frac{\gamma_1}{R}\int_{\Omega^{u}_{n,k}} A(x,u_n)\vert\nabla u_n\vert^{p_1} dx\longrightarrow 0,\\
&\left\vert\int_{\Omega^{v}_{n,k}} B_v(x,v_n) z \vert\nabla v_n\vert^{p_2}dx\right\vert
\leq \frac{\gamma_2}{R}\int_{\Omega^{v}_{n,k}} B(x,v_n)\vert\nabla v_n\vert^{p_2} dx\longrightarrow 0.
\end{split}
\]
On the other hand, it results
\[
\begin{split}
\left|\int_{\Omega}(G_u(x, u_n, v_n) - G_u(x, T_k u_n, T_k v_n)) w dx\right|
\ \le\ &\int_{\Omega}|G_u(x, u_n, v_n) - G_u(x,u,v)| dx \\
& + \int_{\Omega}|G_u(x,T_k u_n,T_k v_n)  - G_u(x,u,v)| dx,\\
\left|\int_{\Omega}(G_v(x,u_n,v_n) - G_v(x,T_k u_n,T_k v_n)) z dx\right|
\ \le\ & \int_{\Omega}|G_v(x,u_n,v_n) - G_v(x,u,v)| dx\\
&+ \int_{\Omega}|G_v(x,T_k u_n,T_k v_n) - G_v(x,u,v)| dx.
\end{split}
\]
We note that from $(g_0)$ and \eqref{3.18}, respectively \eqref{Tkae}, we have
\[
\begin{split}
&G_u(x, u_n, v_n) \longrightarrow G_u(x,u,v) \quad \hbox{and}\quad 
G_v(x, u_n, v_n) \longrightarrow G_v(x,u,v) \qquad \hbox{a.e. in $\Omega$,}\\
&G_u(x,T_k u_n,T_k v_n) \longrightarrow G_u(x,u,v) \quad \hbox{and}\quad 
G_v(x,T_k u_n,T_k v_n) \longrightarrow G_v(x,u,v)
\qquad \hbox{a.e. in $\Omega$,}
\end{split}
\]
while from $(g_1)$ and Young inequality, direct computations imply that
\[
|G_u(x, u_n, v_n)| \le C(1+ |u_n|^{q_1} + |v_n|^{s_4}) \quad \hbox{and}\quad
|G_v(x, u_n, v_n)| \le C(1+ |u_n|^{s_6} + |v_n|^{q_2})
\qquad \hbox{a.e. in $\Omega$,}
\]
with $s_4$ as in \eqref{crits11} and $s_6$ as in \eqref{crits21},
and 
\[
|G_u(x,T_ku_n,T_k v_n)| \le C \quad \hbox{and}\quad
|G_v(x,T_ku_n,T_kv_n)| \le C
\qquad \hbox{a.e. in $\Omega$;}
\]
thus, \eqref{crits11}, \eqref{crits21}, \eqref{3.17} and \cite[Theorem 4.9]{Br} 
allow us to apply the Dominated Convergence Theorem 
so that we obtain
\[
\begin{split}
&\int_{\Omega}|G_u(x, u_n, v_n) - G_u(x,u,v)| dx \longrightarrow 0, \quad
\int_{\Omega}|G_u(x,T_k u_n,T_k v_n)  - G_u(x,u,v)| dx \longrightarrow 0,\\
& \int_{\Omega}|G_v(x,u_n,v_n) - G_v(x,u,v)| dx \longrightarrow 0, \quad
\int_{\Omega}|G_v(x,T_k u_n,T_k v_n) - G_v(x,u,v)| dx \longrightarrow 0.
\end{split}
\]
So, summing up, from \eqref{sumup} and \eqref{sumup2}, 
all the previous limits imply that
\begin{equation}   	\label{epskn}
\begin{split}
\left\vert\frac{\partial\J}{\partial u}(\T_k(u_n,v_n))[w]\right\vert
&\leq \varepsilon_{k,n} +\left\vert\int_{\Omega^{u}_{n, k}} A(x, u_n) 
\vert\nabla u_n\vert^{p_1-2}\nabla u_n\cdot\nabla w dx\right\vert,\\
\left\vert\frac{\partial\J}{\partial v}(\T_k(u_n,v_n))[z]\right\vert
&\leq \varepsilon_{k,n} +\left\vert\int_{\Omega^{v}_{n,k}} B(x,v_n) 
\vert\nabla v_n\vert^{p_2-2}\nabla v_n\cdot\nabla z dx\right\vert,
\end{split}
\end{equation}
where both $(\eps_{k,n})_n$ represent suitable infinitesimal sequences
independent of $w$, respectively $z$.\\
Using the same notation introduced in \textsl{Step 2}, we evaluate 
the last integrals in \eqref{epskn} by reasoning as in the proof 
of \textsl{Step 3} in \cite[Proposition 4.6]{CP2} but
taking new test functions and passing to the limit in 
\[
\begin{split}
&\frac{\partial\J}{\partial u}(\T_k(u_n,v_n))[w R^{+}_{k}u_n],\quad
\frac{\partial\J}{\partial u}(\T_k(u_n,v_n))[w R^{+}_{k-1}u_n], \\
&\frac{\partial\J}{\partial u}(\T_k(u_n,v_n))[w R^{-}_{k}u_n],\quad
\frac{\partial\J}{\partial u}(\T_k(u_n,v_n))[w R^{-}_{k-1}u_n],
\end{split}
\]
respectively
\[
\begin{split}
& \frac{\partial\J}{\partial v}(\T_k(u_n,v_n))[z R^{+}_{k}v_n],\quad
\frac{\partial\J}{\partial v}(\T_k(u_n,v_n))[z R^{+}_{k-1}v_n],\\
& \frac{\partial\J}{\partial v}(\T_k(u_n,v_n))[z R^{-}_{k}v_n],\quad
\frac{\partial\J}{\partial v}(\T_k(u_n,v_n))[z R^{-}_{k-1}v_n].
\end{split}
\]
Hence, we are able to claim that \eqref{step33} hold.\\
At last, direct computations imply that
\[
\begin{split}
\J(\T_k(u_n,v_n)) 
=\ & \J(u_n, v_n) -\frac{1}{p_1}\int_{\Omega_{n, k}^{u}} A(x, u_n) \vert\nabla u_n\vert^{p_1} dx
-\frac{1}{p_2}\int_{\Omega_{n, k}^{v}} B(x,v_n) \vert\nabla v_n\vert^{p_2} dx\\  
&+ \int_{\Omega}(G(x, u_n, v_n) - G(x, T_k u_n, T_k v_n)) dx,
\end{split}
\]
where from \eqref{Gsigmamax}, \eqref{minmaxPerera1}, \eqref{3.17}, \eqref{3.18},
\eqref{Tkae}, and again the Dominated Convergence Theorem
we have 
\[
\int_{\Omega}(G(x, u_n, v_n) - G(x, T_k u_n, T_k v_n)) dx \longrightarrow 0.
\]
Thus, \eqref{step3.2} follows from \eqref{3.15} and \eqref{Agradto0}.\\
{\sl Step 4.} 
By following some ideas introduced in \cite{AB1} and considering 
the real map $\psi(t) = t {\rm e}^{\eta t^2}$, where $\eta$ can be fixed in a suitable way, 
in particular by applying the same arguments developed in the proof of \cite[Proposition 3.6]{CPP2015}
in order to estimate $\frac{\partial\J}{\partial u}(\T_k(u_n,v_n))[\psi(T_k u_n - u)]$,
 we prove that
\begin{equation}    \label{Tkun_u}
\Vert T_k u_n - u\Vert_{W_1}\rightarrow 0.
\end{equation}
Moreover, reasonig in the same way but considering 
$\frac{\partial\J}{\partial v}(\T_k(u_n,v_n))[\psi(T_k v_n - v)]$,
we have also 
\begin{equation}    \label{Tkvn_v}
\Vert T_k v_n - v\Vert_{W_2}\rightarrow 0.
\end{equation}
Then, condition $(i)$ follows from \eqref{normRk}, \eqref{Tkun_u} and \eqref{Tkvn_v}.\\
{\sl Step 5.} 
By means of Proposition \ref{smooth1} applied to the uniformly bounded sequence $(\T_k(u_n,v_n))_n$, 
from \eqref{Tkae}, \eqref{Tkun_u} and \eqref{Tkvn_v} it follows that
\[
\J(\T_k(u_n, v_n))\longrightarrow \J(u,v)\quad \hbox{and}\quad
\Vert d\J(\T_k(u_n, v_n)) - d\J(u, v)\Vert_{X^{\prime}}\longrightarrow 0.
\]
Hence, $(ii)$ follows from \eqref{step3.1} and \eqref{step3.2}.
\end{proof}


\section{Main results}    \label{sec_main}

In order to introduce a suitable decomposition of the space $X$, 
we recall that,
if $p > 1$ but $p \ne 2$, the spectral properties of the operator $-\Delta_p$ 
in $W^{1,p}_0(\Omega)$ are still mostly unknown. In particular,
with respect to the semilinear case, the drawback in using its known eigenvalues  
is that, for the Banach space $W^{1,p}_0(\Omega)$, their use does not provide a decomposition 
having properties similar to that of the Hilbert space $H^1_0(\Omega)$ 
by means of the eigenfunctions of $-\Delta$ on $\Omega$ with null homogeneous Dirichlet data.  

Here, we consider the sequence of pseudo--eigenvalues of the operator $-\Delta_p$ 
in $W^{1,p}_0(\Omega)$ as introduced in \cite[Section 5]{CP2}
 as a suitable decomposition of the Sobolev space $W^{1,p}_0(\Omega)$ occurs, 
so that it turns out to be the classical one for $p=2$. 

In order to present this definition, firstly let us recall that if $V\subseteq X$ is 
a closed subspace of a Banach space $X$, a subspace
$Y\subseteq X$ is a {\sl topological complement} of $V$, briefly $X=V\oplus Y$,
 if $Y$ is closed and every $x\in X$ can be uniquely written as $w+y$, with $w\in V$ and $y\in Y$; 
furthermore, the projection operators onto $V$ and $Y$ are (linear and) continuous
(see, e.g., \cite[p. 38]{Br}). 
When  $X=V\oplus Y$ and $V$ has finite dimension, we say that $Y$ has finite codimension, 
with $\codim Y=\dim V$. 

Now, as in \cite[Section 5]{CP2},
for $i\in\lbrace 1, 2\rbrace$, we start from $\lambda_{i,1}$,
first eigenvalue of $-\Delta_{p_i}$ in $W_i$, 
which is characterized as
\begin{equation}\label{primo}
\lambda_{i,1}=\displaystyle{\inf_{\xi\in W_i\setminus\{0\}}
\frac{\int_\Omega |\nabla \xi|^{p_i} dx}{\int_\Omega |\xi|^{p_i} dx}},
\end{equation}
and is strictly positive, simple, isolated and has a unique eigenfunction $\varphi_{i,1}$
such that
\begin{equation}\label{eig}
\varphi_{i,1} > 0 \;\hbox{a.e. in $\Omega$,}\quad \varphi_{i,1} \in L^\infty(\Omega)
\quad\hbox{and}\quad |\varphi_{i,1}|_{p_i}=1
\end{equation}
(see, e.g., \cite{Lin}).
Then, we have the existence of a sequence $(\lambda_{i,m})_m$ such that
\begin{equation} \label{rn0}
0 < \lambda_{i,1} < \lambda_{i,2} \le \dots \le \lambda_{i,m} \le \dots\qquad
\hbox{and}\qquad
\lambda_{i,m}\nearrow +\infty \quad \hbox{if $m\to+\infty$,}
\end{equation}
with corresponding functions $(\psi_{i,m})_m$ such that $\psi_{i,1}\equiv \varphi_{i,1}$ 
and $\psi_{i,m}\ne \psi_{i,j}$ if $m\ne j$. 
They generate the whole space $W_i$, are also in $L^\infty(\Omega)$,
hence in $X_i$, and are such that
\[
W_i = V_{i,m}\oplus Y_{i,m} \quad \hbox{ for all } m\in \N, 
\]
where $V_{i,m} = {\rm span}\{\psi_{i,1},\ldots,\psi_{i,m}\}$ 
and its complement $Y_{i,m}$ in $W_i$ can be explicitely described.
Moreover, for all $m\in \N$ 
on the infinite dimensional subspace $Y_{i,m}$ the following inequality holds:
\begin{equation}\label{lambdan+1}
\lambda_{i,m+1}\ \int_\Omega|w|^{p_i} dx\ \leq\ \int_\Omega |\nabla w|^{p_i} dx \quad
\hbox{ for all }  w\in Y_{i,m}
\end{equation}
(cf. \cite[Proposition 5.4]{CP2}).

Thus, for all $m\in \N$, in our setting \eqref{Wdefn1} we have
\[
W = (V_{1,m}\times V_{2,m})\oplus (Y_{1,m}\times Y_{2,m})
\]
and then, since $V_{i,m}$ is also a finite dimensional subspace of $X_i$, 
\[
X = (V_{1,m}\times V_{2,m})\oplus(Y_m^{X_1}\times Y_m^{X_2}),
\]
with
\[
Y_m^{X_i} = Y_{i,m}\cap L^{\infty}(\Omega)\subset X_i\quad \hbox{for $i \in\lbrace 1, 2\rbrace$.}
\] 
We note that it results $X_i = V_{i,m}\oplus Y_m^{X_i}$ with
\[
{\rm dim} V_{i,m} = m \quad \mbox{ and } \quad {\rm codim} Y_m^{X_i} = m.
\] 

Now, we can state our main results.

\begin{theorem}    \label{ThExist}
Taking $p_1$, $p_2 > 1$, let $A(x, u)$ and $B(x, v)$ be two given real functions defined in $\Omega\times \R$
such that conditions $(h_0)$--$(h_4)$ hold. 
Moreover, suppose that the real function $G(x,u,v)$,
defined in $\Omega\times \R^2$, satisfies hypotheses $(g_0)$--$(g_2)$. 
In addition, we assume that
\begin{enumerate}
\item [$(g_3)$] $\;\quad\displaystyle\limsup_{(u, v)\to 0} \frac{G(x, u, v)}{\vert u\vert^{p_1} 
+ \vert v\vert^{p_2}}\ <\  
\mu_0\min\left\lbrace \frac{\lambda_{1,1}}{p_1}, \frac{\lambda_{2,1}}{p_2}\right\rbrace \quad$ 
uniformly a.e. in $\Omega$,
\end{enumerate}
with $\lambda_{i,1}$, $i \in \{1,2\}$, as in \eqref{primo}.
Then, functional $\J$ in \eqref{functional} possesses at least one nontrivial critical point
in $X$; hence, problem \eqref{euler} admits a nontrivial weak bounded solution.
\end{theorem}

\begin{theorem}     \label{ThMolt}
Taking $p_1$, $p_2 > 1$, suppose that $A(x, u)$, $B(x, v)$ and $G(x,u,v)$ satisfy
hypotheses $(h_0)$--$(h_4)$, $(g_0)$--$(g_2)$. Moreover, assume also that
\begin{itemize}
\item[$(h_5)$]  $\; A(x, \cdot), B(x, \cdot)$ are even in $\R$ for a.e. $x\in\Omega$;
\item[$(g_4)$] $\;\quad\displaystyle\liminf_{|(u, v)|\to +\infty} 
\frac{G(x, u, v)}{\vert u\vert^{\frac{1}{\theta_1}} + \vert v\vert^{\frac{1}{\theta_2}}}\ >\ 0\ $
uniformly a.e. in $\Omega$;
\item[$(g_5)$] $\; G(x,\cdot, \cdot)$ is even in $\R^2$ for a.e. $x\in\Omega$.
\end{itemize}
Then, the even functional $\J$ in \eqref{functional} possesses a sequence of critical 
points $((u_m,v_m))_m \subset X$ such that $\J(u_m,v_m)\nearrow +\infty$; hence,
problem \eqref{euler} admits infinitely many distinct weak bounded solutions.
\end{theorem}

\begin{remark}
Subcritical growth conditions \eqref{crit_expi} are stronger than the 
``classical'' ones required for Laplacian coupled systems in \cite{BdF}.
In our setting, we need it for proving that weak limits of $(CPS)$--sequences 
in the product Sobolev space $W$ belong also to $L$, i.e. they are bounded functions  
(see \textsl{Step 2} in the proof of Proposition \ref{PropwCPS}).
\end{remark}

Some corollaries to previous Theorem \ref{ThMolt} 
can be stated, both when condition $(g_4)$ is replaced 
by a stronger assumption, which is easier to verify,
and when such a theorem is applied to the special cases obtained by coupling 
Example \ref{ex0} with Example \ref{ex1}, respectively Example \ref{ex2}. 

\begin{corollary}     \label{ThMolt1}
Taking $p_1$, $p_2 > 1$, suppose that $A(x, u)$, $B(x, v)$ and $G(x,u,v)$ satisfy
hypotheses $(h_0)$--$(h_5)$, $(g_0)$--$(g_2)$ and $(g_5)$. Moreover, assume also that
\begin{itemize}
\item[$(g_6)$] $\; \inf\{G(x,w,z):\ x \in \Omega,\ (w,z) \in \R^2 \ 
\hbox{such that $|(w,z)|=R$}\} > 0$, with $R$ as in $(g_2)$;  
\item[$(g_7)$] $\; \theta_1=\theta_2$, with $\theta_1$, $\theta_2$ as in $(h_4)$ and $(g_2)$.
\end{itemize}
Then, the even functional $\J$ in \eqref{functional} possesses a sequence of critical 
points $((u_m,v_m))_m \subset X$ such that $\J(u_m,v_m)\nearrow +\infty$; hence,
problem \eqref{euler} admits infinitely many distinct weak bounded solutions.
\end{corollary}

\begin{corollary}\label{cor1}
Let us consider problem \eqref{euler} with $p_1$, $p_2 > 1$, 
and $A(x,u)$, $B(x,v)$ as in \eqref{ex01} such that conditions 
\eqref{ex05}--\eqref{ex03} hold.
Moreover, consider $G(x,u,v)$ defined as in \eqref{ex11} and such that
\eqref{ex13}--\eqref{ex17}, \eqref{ex12} are satisfied.
If
\begin{equation}     \label{cor11}
p_1+\gamma_1 < q_1 < p^*_1,\qquad p_2+\gamma_2 < q_2 < p^*_2,
\end{equation}
\begin{equation}     \label{cor12}
\gamma_4\ \frac{q_1-1}{q_1-\gamma_3}\ <\ \frac{p_1}{N}\left(1 - \frac{1}{p_1^*}\right) p^*_2,\qquad
\gamma_3\ \frac{q_2-1}{q_2-\gamma_4}\ <\ \frac{p_2}{N}\left(1 - \frac{1}{p_2^*}\right) p^*_1,
\end{equation}
then functional
\[
\begin{split}
\J_1(u,v) =\ &\frac{1}{p_1} \int_\Omega(A_1(x) + A_2(x) |u|^{\gamma_1}) |\nabla u|^{p_1} dx 
+ \frac{1}{p_2} \int_\Omega(B_1(x) + B_2(x) |v|^{\gamma_2}) |\nabla v|^{p_2} dx \\
&- \int_\Omega(|u|^{q_1} + c_* |u|^{\gamma_3} |v|^{\gamma_4} + |v|^{q_2})dx 
\end{split}
\]
has infinitely many critical points in $X$.
\end{corollary}

\begin{corollary}\label{cor2}
Let us consider problem \eqref{euler} with $p_1$, $p_2 > 1$, 
and $A(x,u)$, $B(x,v)$ as in \eqref{ex01} such that conditions 
\eqref{ex05}--\eqref{ex03} hold.
Moreover, consider $G(x,u,v)$ defined as in \eqref{ex21} and such that
\eqref{ex26} is satisfied.
If
\begin{equation}     \label{cor21}
p_1+\gamma_1 < \gamma_3 < q_1 < p^*_1,\qquad
p_2+\gamma_2 < \gamma_4 < q_2 < p^*_2,
\end{equation}
\begin{equation}     \label{cor22}
\gamma_3\ <\ \frac{p_2}{N}\left(1 - \frac{1}{p_2^*}\right) p^*_1,\qquad
\gamma_4\ <\ \frac{p_1}{N}\left(1 - \frac{1}{p_1^*}\right) p^*_2,
\end{equation}
then functional
\[
\begin{split}
\J_2(u,v) =\ &\frac{1}{p_1} \int_\Omega(A_1(x) + A_2(x) |u|^{\gamma_1}) |\nabla u|^{p_1} dx 
+ \frac{1}{p_2} \int_\Omega(B_1(x) + B_2(x) |v|^{\gamma_2}) |\nabla v|^{p_2} dx \\
&- \int_\Omega(|u|^{q_1} + |u|^{\gamma_3} \lg(v^2+1) +\lg(u^2+1) |v|^{\gamma_4} + |v|^{q_2})dx 
\end{split}
\]
has infinitely many critical points in $X$.
\end{corollary}

Now, we can prove our main results. To this aim,
for simplicity we assume that 
\begin{equation}\label{level0}
\int_\Omega G(x,0,0) dx = 0, 
\end{equation}
otherwise we can replace functional $\J(u,v)$ in \eqref{functional}
with $\J(u,v) + \int_\Omega G(x,0,0) dx$ which has the same differential
$d\J(u,v)$ in \eqref{diff}.

\begin{proof}[Proof of Theorem \ref{ThExist}]
Without loss of generality, we suppose $1 < p_i < N$ for both $i\in\{1,2\}$,
otherwise the proof is simpler. \\
Moreover,
as $\mu_0\min\left\lbrace \frac{\lambda_{1,1}}{p_1}, \frac{\lambda_{2,1}}{p_2}\right\rbrace >0$,
from $(g_3)$ we can take $\bar{\lambda} > 0$ such that
\begin{equation}   \label{lambdasegnato}
\limsup_{(u, v)\to 0} \frac{G(x, u, v)}{\vert u\vert^{p_1} +\vert v\vert^{p_2}} 
\ < \ \bar{\lambda}\ <\ 
\mu_0\min\left\lbrace \frac{\lambda_{1,1}}{p_1}, \frac{\lambda_{2,1}}{p_2}\right\rbrace.
\end{equation}
We claim that \eqref{Gsigmamax} and \eqref{lambdasegnato} imply the
existence of a suitable constant $\sigma^* >0$ such that
\begin{equation}   \label{Gcompleta}
G(x, u, v) \le \bar{\lambda}(\vert u\vert^{p_1} + \vert v\vert^{p_2}) 
+ \sigma^* (\vert u\vert^{\overline{q}_1} +\vert v\vert^{\overline{q}_2})
\quad \mbox{ for a.e. $x\in\Omega$, all $(u,v) \in \R^2$.}
\end{equation}
In fact, from one hand inequality \eqref{lambdasegnato} implies that a
radius $\varrho^* > 0$ exists such that
\[
G(x, u, v) \le \bar{\lambda}(\vert u\vert^{p_1} + \vert v\vert^{p_2}) 
\quad \mbox{ for a.e. $x\in\Omega$ if $|(u,v)| < \varrho^*$.}
\]
On the other hand, if $|(u,v)| \ge \varrho^*$, then either $|u| \ge \frac{\varrho^*}{2}$,
hence $1 \le \frac{2}{\varrho^*}|u|$ and
$|u| \le \left(\frac{2}{\varrho^*}\right)^{\overline{q}_1-1}|u|^{\overline{q}_1}$,
or $|v| \ge \frac{\varrho^*}{2}$,
hence $1 \le \frac{2}{\varrho^*}|v|$ and $|v| \le \left(\frac{2}{\varrho^*}\right)^{\overline{q}_2-1}|v|^{\overline{q}_2}$.
So, analyzing all the possible cases,  
direct computations imply that 
\[
1 \le \frac{2}{\varrho^*}(|u| + |v|)\quad \hbox{and}\quad 
\vert u\vert +\vert v\vert \le 2 \left(\frac{2}{\varrho^*}\right)^{\overline{q}_1-1}|u|^{\overline{q}_1}
+ 2 \left(\frac{2}{\varrho^*}\right)^{\overline{q}_2-1}|v|^{\overline{q}_2}
\quad \hbox{if $|(u,v)| \ge \varrho^*$.}
\]
Whence, \eqref{Gcompleta} follows from \eqref{Gsigmamax}.\\
Now, from definition \eqref{functional}, hypothesis $(h_2)$ 
and estimate \eqref{Gcompleta}, we have that
\[
\begin{split}
\J(u, v)\ \geq\ &\frac{\mu_0}{p_1}\int_{\Omega} \vert\nabla u\vert^{p_1} dx +
\frac{\mu_0}{p_2}\int_{\Omega}\vert\nabla v\vert^{p_2} dx
- \bar{\lambda} \int_{\Omega} \vert u\vert^{p_1} dx -
\bar{\lambda} \int_{\Omega} \vert v\vert^{p_2} dx\\
&- \sigma^* \int_{\Omega} \vert u\vert^{\overline{q}_1} dx 
- \sigma^* \int_{\Omega} \vert v\vert^{\overline{q}_2} dx ,
\end{split}
\]
or better, from \eqref{Sobpi}, \eqref{minmaxPerera1} and \eqref{primo} it follows that
\begin{equation}   \label{geo1}
\begin{split}
\J(u, v)\ \geq\ & \Vert u\Vert_{W_1}^{p_1}\left(\frac{\mu_0}{p_1} -\frac{\bar{\lambda}}{\lambda_{1,1} }
- \sigma^* \tau_{1,\overline{q}_1}^{\overline{q}_1} \Vert u\Vert_{W_1}^{\overline{q}_1 -p_1}\right)\\
& + \Vert v\Vert_{W_2}^{p_2}\left(\frac{\mu_0}{p_2} -\frac{\bar{\lambda}}{\lambda_{2,1}} 
- \sigma^* \tau_{2,\overline{q}_2}^{\overline{q}_2}\Vert v\Vert_{W_2}^{\overline{q}_2- p_2}\right)\\
\geq\ & \Vert u\Vert_{W_1}^{p_1}\left(\frac{\mu_0}{p_1} -\frac{\bar{\lambda}}{\lambda_{1,1} }
- \sigma^* \tau_{1,\overline{q}_1}^{\overline{q}_1} \Vert(u,v)\Vert_{W}^{\overline{q}_1 -p_1}\right)\\
& + \Vert v\Vert_{W_2}^{p_2}\left(\frac{\mu_0}{p_2} -\frac{\bar{\lambda}}{\lambda_{2,1}} 
- \sigma^* \tau_{2,\overline{q}_2}^{\overline{q}_2}\Vert(u,v)\Vert_{W}^{\overline{q}_2- p_2}\right)
\end{split}
\end{equation}
for all $(u,v)\in X$. 
Thus, from \eqref{minmaxPerera1} and \eqref{lambdasegnato} we have that a radius $R_0 > 0$ 
exists such that
\begin{equation}   \label{geo21}
\frac{\mu_0}{p_1} -\frac{\bar{\lambda}}{\lambda_{1,1} }
- \sigma^* \tau_{1,\overline{q}_1}^{\overline{q}_1} R_0^{\overline{q}_1 -p_1}\ >\ 0,
\qquad
\frac{\mu_0}{p_2} -\frac{\bar{\lambda}}{\lambda_{2,1}} 
- \sigma^* \tau_{2,\overline{q}_2}^{\overline{q}_2}R_0^{\overline{q}_2- p_2}\ >\ 0;
\end{equation}
hence, \eqref{geo1}, \eqref{geo21} and direct computations imply 
the existence of a constant $\varrho_0 > 0$ such that
\begin{equation}   \label{geo3}
\J(u, v)\ \geq\ \varrho_0\quad \hbox{if $\Vert(u,v)\Vert_{W} = R_0$.}
\end{equation}
At last, taking for example $\varphi_{1,1} \in X_1$, with $\varphi_{1,1}$ first eigenvalue 
of $-\Delta_{p_1}$ in $W_1$ so that \eqref{eig} holds with $i = 1$, and $R \ge 1$ as in
\eqref{rmkPerera}, we can consider 
\begin{equation}   \label{geo6}
\bar{u}(x) = R \frac{\varphi_{1,1}(x)}{|\varphi_{1,1}(x)|} 
\quad \hbox{for a.e. $x \in \Omega$.}
\end{equation}
Then, fixing any $t \ge 1$, from \eqref{functional}, \eqref{Atheta}, \eqref{rmkPerera}
and direct computations it follows that
\[
\J(t\bar{u},0) 
\le\  t^{p_1} \frac{a_1}{p_1} \Vert\bar{u}\Vert_{W_1}^{p_1} +
t^{\frac{1}{\theta_1}(1-\mu_2)} \frac{a_2}{p_1} |\bar{u}|_\infty^{\frac1{\theta_1}(1 - \theta_1p_1-\mu_2)} \|\bar{u}\|_{W_1}^{p_1} 
 - t^{\frac{1}{\theta_1}} \int_{\Omega}h_1(x)|\bar{u}|^{\frac{1}{\theta_1}} dx,
\]
with $0 < \int_{\Omega}h_1(x)|\bar{u}|^{\frac{1}{\theta_1}} dx < +\infty$ 
as $\bar{u}$, $h_1 \in L^\infty(\Omega)$ with both 
$h_1(x) > 0$ (see Remark \ref{rmkPerera0}) and $\bar{u}(x) > 0$ (from \eqref{eig}
and \eqref{geo6}) for a.e. $x \in \Omega$.
Hence, \eqref{thetai} implies that 
\[
\J(t\bar{u}, 0)\longrightarrow -\infty \quad \mbox{ as } t\to +\infty.
\]
Thus, $\bar{e}_1\in X_1$ exists such that 
\begin{equation}   \label{geo4}
\Vert(\bar{e}_1,0)\Vert_{W} > R_0\quad \hbox{and}\quad 
\J(\bar{e}_1,0) <\varrho_0.
 \end{equation}
Finally, from \eqref{level0} we have that $\J(0,0) = 0$ and, 
summing up, Propositions \ref{smooth1} and \ref{PropwCPS} together 
with \eqref{geo3}, \eqref{geo4} allow us to apply Theorem \ref{mountainpass}.
So, a critical point $(u^*,v^*)\in X$ exists such that
$\J(u^*,v^*) \ge \varrho_0 > 0$.
\end{proof}

Now, we want to prove Theorem \ref{ThMolt} by applying Corollary \ref{multiple}. 
To this aim, we need the following ``geometric'' estimates.
 
\begin{proposition}\label{geo2}
Assume that hypotheses $(h_0)$--$(h_2)$, $(h_4)$, $(g_0)$--$(g_2)$ and $(g_4)$ hold.
Then, taking any finite dimensional subspace $V$ of $X$, there exists
$R_V > 0$ such that
\[
\J(u,v) \le 0\qquad \hbox{for all $(u,v) \in V$, $\|(u,v)\|_X \ge R_V$.}
\]
In particular, $\J$ is bounded from above in $V$.
\end{proposition}

\begin{proof}
Considering \eqref{Xnorms}, from \eqref{functional}, \eqref{Atheta}, \eqref{Btheta}
it follows that
\[
\begin{split}
\J(u,v) \le &\ \frac{a_1}{p_1} \|u\|^{p_1}_{W_1} 
+ \frac{a_2}{p_1} |u|_{\infty}^{\frac{1}{\theta_1}(1-p_1\theta_1 -\mu_2)} \|u\|^{p_1}_{W_1}
+ \frac{b_1}{p_2} \|v\|^{p_2}_{W_2} 
+ \frac{b_2}{p_2} |v|_{\infty}^{\frac{1}{\theta_2}(1-p_2\theta_2 -\mu_2)} \|v\|^{p_2}_{W_2}\\
&\ - \int_\Omega G(x,u,v) dx\\
\le &\ \frac{a_1}{p_1} \|u\|^{p_1}_{X_1} 
+ \frac{a_2}{p_1} \|u\|_{X_1}^{\frac{1}{\theta_1}(1-\mu_2)} 
+ \frac{b_1}{p_2} \|v\|^{p_2}_{X_2} 
+ \frac{b_2}{p_2} \|v\|_{X_2}^{\frac{1}{\theta_2}(1 -\mu_2)}
 - \int_\Omega G(x,u,v) dx
\end{split}
\] 
for all $(u,v) \in X$.
Now, from $(g_4)$ we can take $\bar{\bar{\lambda}} \in \R$ such that
\[
\liminf_{|(u, v)|\to +\infty} 
\frac{G(x, u, v)}{\vert u\vert^{\frac{1}{\theta_1}} + \vert v\vert^{\frac{1}{\theta_2}}}\
> \ \bar{\bar{\lambda}}\ >\ 0 \quad \hbox{uniformly a.e. in $\Omega$,}
\]
then $R_1 > 0$ exists such that
\begin{equation}\label{ggeo1}
G(x, u, v)\ \ge \ \bar{\bar{\lambda}}\ 
(\vert u\vert^{\frac{1}{\theta_1}} + \vert v\vert^{\frac{1}{\theta_2}}) 
\quad \hbox{for a.e. $x \in \Omega$ if $|(u,v)| \ge R_1$.}
\end{equation}
Thus, taking $(u,v) \in V$ such that
\begin{equation}\label{suR}
|(u,v)|_\infty > R_1 \qquad \then\qquad 
\meas(\Omega_{R_1}) > 0
\end{equation}
with $\Omega_{R_1} = \{x \in \Omega: |(u(x),v(x))| > R_1\}$.
Then, since \eqref{Gsigma}
implies
\[
\int_{\Omega\setminus \Omega_{R_1}} |G(x,u,v)| dx \le C_1
\]
for a suitable $C_1 > 0$, from \eqref{ggeo1}, \eqref{suR}
and direct computations it follows that 
\[
\begin{split}
\int_{\Omega} G(x,u,v) dx\ &\ge\ \bar{\bar{\lambda}}\
\int_{\Omega_{R_1}} (\vert u\vert^{\frac{1}{\theta_1}} + \vert v\vert^{\frac{1}{\theta_2}}) dx - C_1\\
&\ge\ \bar{\bar{\lambda}}\ \left(\int_{\Omega}\vert u\vert^{\frac{1}{\theta_1}} dx 
+ \int_{\Omega} \vert v\vert^{\frac{1}{\theta_2}}\right) \ -\ C_2
\end{split} 
 \]
for a suitable $C_2 > 0$. 
At last, summing up, as all the norms are equivalent 
in the finite dimensional space $V = V_1 \times V_2$ and the same is in 
both the projections $V_1$ in $X_1$ and $V_2$ in $X_2$, 
if $\|(u,v)\|_{X}$ is large enough then
\eqref{suR} holds and we have that
\[
\J(u,v) \le  \frac{a_1}{p_1} \|u\|^{p_1}_{X_1} 
+ \frac{a_2}{p_1} \|u\|_{X_1}^{\frac{1}{\theta_1}(1-\mu_2)} 
+ \frac{b_1}{p_2} \|v\|^{p_2}_{X_2} 
+ \frac{b_2}{p_2} \|v\|_{X_2}^{\frac{1}{\theta_2}(1 -\mu_2)}
 - C_3 \|u\|_{X_1}^{\frac1{\theta_1}} - C_4 \|v\|_{X_2}^{\frac1{\theta_2}} + C_5
\]
for suitable constants $C_j > 0$; hence, condition \eqref{thetai} implies that
\[
\J(u,v) \to -\infty \quad \hbox{ as $\|(u,v)\|_{X} \to +\infty$}
\]
which completes the proof.
\end{proof}

\begin{proposition}     \label{PropMolt}
Assume that hypotheses $(h_0)$--$(h_2)$, $(h_4)$, $(g_0)$--$(g_2)$ hold.
Then, for any fixed $\varrho\in\R$ there exists $m=m(\varrho)\geq 1$ and $r_m >0$ such that
\[
(u,v)\in Y_m^{X_1}\times Y_m^{X_2}, \quad \Vert (u,v)\Vert_{W} = r_m
\qquad \implies\qquad \J(u, v)\geq \varrho.
\]
\end{proposition}

\begin{proof}
The proof is essentially as in \cite[Proposition 6.4]{CP2}, carefully extended 
to a system. For completeness, we give here all the details and, 
without loss of generality, we suppose $1 < p_i < N$ for both $i \in \{1,2\}$
(otherwise, the proof is simpler).\\
Taking any $(u,v) \in X$,
from \eqref{functional}, hypothesis $(h_2)$, \eqref{Gsigmamax} and direct computations
it follows that
\begin{equation}    \label{J5.3}
\J(u, v)\geq \frac{\mu_0}{p_1}\int_{\Omega}\vert\nabla u\vert^{p_1} dx +
\frac{\mu_0}{p_2}\int_{\Omega}\vert\nabla v\vert^{p_2} dx - 
C_1\int_{\Omega}\vert u\vert^{\overline{q}_1} dx - C_1\int_{\Omega} \vert v\vert^{\overline{q}_2} dx - C_2
\end{equation}
for suitable constants $C_1$, $C_2 > 0$.
Moreover, \eqref{minmaxPerera1} allows us to take $r_1, r_2 >0$ such that
\[
\frac{r_1}{p_1} +\frac{\overline{q}_1-r_1}{p_1^{\ast}} =1 \quad 
\mbox{ and }\quad \frac{r_2}{p_2} +\frac{\overline{q}_2 - r_2}{p_2^{\ast}} = 1;
\]
so, by using a standard interpolation argument and \eqref{Sobpi},
we have that
\[
\begin{split}
&\vert u\vert_{\overline{q}_1}^{\overline{q}_1}\leq \vert u\vert_{p_1^{\ast}}^{\overline{q}_1 - r_1} \vert u\vert_{p_1}^{r_1}
\leq\tau_{1,p_1^{\ast}}^{\overline{q}_1- r_1} \Vert u\Vert_{W_1}^{\overline{q}_1- r_1} \vert u\vert_{p_1}^{r_1}
\quad \hbox{for all $u\in X_1$,} \\
&\vert v\vert_{\overline{q}_2}^{\overline{q}_2}\leq \vert v\vert_{p_2^{\ast}}^{\overline{q}_2 - r_2} \vert v\vert_{p_2}^{r_2}
\leq\tau_{2,p_2^{\ast}}^{\overline{q}_2- r_2}\Vert v\Vert_{W_2}^{\overline{q}_2 - r_2} \vert v\vert_{p_2}^{r_2} 
\quad \hbox{for all $v\in X_2$,}
\end{split}
\]
or better, fixing $m\geq 1$, from \eqref{lambdan+1} it follows that
\begin{equation}   \label{intimes}
\begin{split}
&\vert u\vert_{\overline{q}_1}^{\overline{q}_1}\leq 
\tau_{1,p_1^{\ast}}^{\overline{q}_1- r_1}(\lambda_{1,m+1})^{-\frac{r_1}{p_1}} \Vert u\Vert_{W_1}^{\overline{q}_1}
\qquad \hbox{for all $u\in Y_m^{X_1}$,} \\
&\vert v\vert_{\overline{q}_2}^{\overline{q}_2}\leq\tau_{2,p_2^{\ast}}^{\overline{q}_2- r_2}(\lambda_{2,m+1})^{-\frac{r_2}{p_2}} 
\Vert v\Vert_{W_2}^{\overline{q}_2} 
\qquad \hbox{for all $v\in Y_m^{X_2}$.}
\end{split}
\end{equation}
Then, \eqref{J5.3}, \eqref{intimes} and direct computations imply that
\[
\J(u, v)\geq \bar{\mu}_0 (\Vert u\Vert_{W_1}^{p_1}
+ \Vert v\Vert_{W_2}^{p_2}) - \frac{C_3}{\bar{\lambda}_{m}} (\Vert u\Vert_{W_1}^{\overline{q}_1} 
+ \Vert v\Vert_{W_2}^{\overline{q}_2}) - C_2
\qquad \hbox{for all $(u,v)\in Y_m^{X_1} \times Y_m^{X_2}$}
\]
for a suitable constant $C_3 > 0$ independent of $m$, where we assume
\begin{equation}   \label{rn1}
\bar{\mu}_0 = \min\left\{\frac{\mu_0}{p_1}, \frac{\mu_0}{p_2}\right\},\quad
\bar{\lambda}_{m} = \min\left\{(\lambda_{1,m+1})^{\frac{r_1}{p_1}}, (\lambda_{2,m+1})^{\frac{r_2}{p_2}}\right\}.
\end{equation}
On the other hand, taking
\[
p = \min\lbrace p_1, p_2\rbrace \quad \mbox{ and }\quad q = \max\lbrace\overline{q}_1, \overline{q}_2\rbrace,
\]
direct computations allow us to prove that
\[
\Vert u\Vert_{W_1}^{p_1} + \Vert v\Vert_{W_2}^{p_2} \ge \frac{1}{2^{p-1}}\|(u,v)\|_W^p -1
\quad \mbox{ and }\quad
\Vert u\Vert_{W_1}^{\overline{q}_1} 
+ \Vert v\Vert_{W_2}^{\overline{q}_2} \le 2 \|(u,v)\|_W^q\quad
\hbox{if $\|(u,v)\|_W \ge 2$.}
\]
Thus, summing up, for all $(u,v)\in Y_m^{X_1} \times Y_m^{X_2}$
such that  $\|(u,v)\|_W \ge 2$ we obtain that
\[
\J(u, v)\ \geq\ \|(u,v)\|_W^p \left(\frac{\bar{\mu}_0}{2^{p-1}} 
- \frac{2 C_3}{\bar{\lambda}_{m}} \|(u,v)\|_W^{q-p}\right) - C_4
\]
with $C_4 = C_2+\bar{\mu}_0$.\\
Now, consider $\|(u,v)\|_W = r_m$ with 
\begin{equation}   \label{rn}
r_m = \left(\frac{\bar{\mu}_0\ \bar{\lambda}_m}{2^{p+1} C_3}\right)^{\frac1{q-p}}
\end{equation}
(from \eqref{minmaxPerera1} it is $p<q$).
We note that \eqref{rn} implies 
\[
\frac{\bar{\mu}_0}{2^{p-1}} - \frac{2 C_3}{\bar{\lambda}_{m}} r_m^{q-p} =
\frac{\bar{\mu}_0}{2^{p}},
\]
furthermore again from \eqref{rn} and \eqref{rn0}, \eqref{rn1} we have 
\[
r_m \to +\infty\quad \hbox{if $m \to+\infty$.}
\]
Then, if $r_m \ge 2$, we obtain that
\[
\J(u, v)\ \geq\ \frac{\bar{\mu}_0}{2^p} r_m^{p} - C_4\qquad
\hbox{for all $(u, v)\in Y_m^{X_1}\times Y_m^{X_2}$, 
with $\Vert(u, v)\Vert_{W} = r_m$} 
\]
which gives the thesis if $\varrho > 0$ is fixed and $m$ is large enough.
\end{proof}

\begin{proof}[Proof of Theorem \ref{ThMolt}]
From \eqref{level0} it is $\J(0,0) = 0$. Then, 
fixing any $\varrho >0$, from Proposition \ref{PropMolt} it follows that
$m_\varrho\geq 1$ and $r_{m_\varrho}>0$ exist such that
\[
(u, v)\in Y_{m_\varrho}^{X_1}\times Y_{m_\varrho}^{X_2}, \quad
\Vert(u, v)\Vert_{W} = r_{m_\varrho}\qquad \implies\qquad \J(u, v)\geq\varrho;
\]
moreover, taking $m> m_\varrho$, the $m$--dimensional space $V_m$ is such that 
${\rm codim} Y_{m_\varrho} < {\rm dim} V_m$,
thus, Proposition \ref{geo2} implies that 
assumption $({\cal H}_\varrho)$ is verified
with 
\[
\M_\varrho = \{(u,v) \in X :\ \Vert(u, v)\Vert_{W} = r_{m_\varrho}\}.
\]
Hence, from Propositions \ref{smooth1} and
\ref{PropwCPS} we have that Corollary \ref{multiple} applies to $\J$ in $X$ 
and a sequence of diverging critical levels exists.
\end{proof}

In order to prove Corollary \ref{ThMolt1}, the following lemma is useful.

\begin{lemma} 
If hypotheses $(g_0)$ and $(g_2)$ hold, 
then, taking any $(u,v) \in \R^2$ such that $|(u,v)|\ge R$ we have that 
\begin{equation}     \label{htheta}
G(x,t^{\theta_1}u,t^{\theta_2}v)\ \geq\ t\ G(x,u,v)\quad 
\mbox{ for a.e. } x\in\Omega, \mbox{ all } t \ge 1.
\end{equation}
\end{lemma}

\begin{proof}
Taking $(u,v) \in \R^2$ such that $|(u,v)|\ge R$ 
and any $t \ge 1$, $(g_0)$, $(g_2)$ and direct computations 
imply that
\[\begin{split}
\frac{d}{d t}G(x,t^{\theta_1}u,t^{\theta_2}v)\ &=
\frac{1}{t} \big(\theta_1 G_u(x,t^{\theta_1}u,t^{\theta_2}v) t^{\theta_1}u +
\theta_2 G_v(x,t^{\theta_1}u,t^{\theta_2}v) t^{\theta_2}v\big)\\
&\ge \frac{1}{t} G(x,t^{\theta_1}u,t^{\theta_2}v) > 0
\end{split}
\]
for a.e. $x \in \Omega$. Hence, fixing $t \ge 1$ and by integration in $[1,t]$
we obtain that 
\[
\lg\left(\frac{G(x,t^{\theta_1}u,t^{\theta_2}v)}{G(x,u,v)}\right) \ge \lg t
\]
which gives \eqref{htheta}.
\end{proof}

\begin{proof}[Proof of Corollary \ref{ThMolt1}]
The proof is a direct consequence of Theorem \ref{ThMolt}  
if we prove that $(g_4)$, with $\theta = \theta_1 = \theta_2$
due to $(g_7)$, follows from $(g_6)$.\\
To this aim, since $(g_6)$ holds, 
we denote 
\[
C_R = \inf\{G(x,w,z):\ \hbox{for a.e.}\ x \in \Omega,\ (w,z) \in \R^2 \ \hbox{such that $|(w,z)|=R$}\} > 0
\]
with $R > 0$ as in $(g_2)$.
Then, fixing any $(u,v) \in \R^2$ such that $|(u,v)| \ge R$,
by applying \eqref{htheta} with 
\[
t = \left(\frac{|(u,v)|}{R}\right)^{\frac1\theta}\ \ge \ 1,
\]
direct computations imply
\[
G(x,u,v) \ \ge\ \left(\frac{|(u,v)|}{R}\right)^{\frac1\theta} G\left(x,\frac{R u}{|(u,v)|},\frac{R v}{|(u,v)|}\right)
\ \ge\ \frac{C_R}{R^{\frac1\theta}} |(u,v)|^{\frac1\theta}\quad \hbox{for a.e.}\ x \in \Omega.
 \]
Hence, we obtain
\[
\frac{G(x, u, v)}{\vert u\vert^{\frac{1}{\theta}} + \vert v\vert^{\frac{1}{\theta}}}\
\ge\ \frac{C_R}{R^{\frac1\theta}} \
\frac{|(u,v)|^{\frac1\theta}}{\vert u\vert^{\frac{1}{\theta}} + \vert v\vert^{\frac{1}{\theta}}}
\ \ge\ \frac{C_R}{2R^{\frac1\theta}} > 0
\quad \hbox{for a.e. $x \in \Omega$ if $|(u,v)| \ge R$,}
\]
which implies $(g_4)$.
\end{proof}

\begin{proof}[Proof of Corollary \ref{cor1}]
From \eqref{cor11}, taking $\theta_i > 0$ such that
\[
p_i+\gamma_i < \frac1{\theta_i} \le q_i,\quad i \in \{1,2\}, 
\]
we have that \eqref{ex06} and \eqref{ex15} hold.
Moreover, \eqref{cor11} implies \eqref{crit_exp} while
from \eqref{ex16} and \eqref{cor12} it follows \eqref{crit_expi}.
At last, we note that \eqref{ex12} gives
\[
\frac{G(x,u,v)}{|u|^{\frac1{\theta_1}} + |v|^{\frac1{\theta_2}}} 
\ge \frac{|u|^{q_1} + |v|^{q_2}}{|u|^{\frac1{\theta_1}} + |v|^{\frac1{\theta_2}}}
\quad \hbox{if $(u,v) \ne 0$.}
\]
Then, $(g_4)$ holds as direct computations imply
\begin{equation}\label{corbis} 
\frac{|u|^{q_1} + |v|^{q_2}}{|u|^{\frac1{\theta_1}} + |v|^{\frac1{\theta_2}}} \ge \frac12
\quad \hbox{if $|(u,v)| \ge 2$.}
\end{equation}
Thus, from Examples \ref{ex0} and \ref{ex1}, it follows that
Theorem \ref{ThMolt} applies to $\J_1$ in $X$.
\end{proof}

\begin{proof}[Proof of Corollary \ref{cor2}]
From \eqref{cor21}, taking $\theta_1$, $\theta_2 > 0$ such that
\begin{equation}\label{cor21bis}
p_1+\gamma_1 < \frac1{\theta_1} \le \gamma_3, \qquad
p_2+\gamma_2 < \frac1{\theta_2} \le \gamma_4,
\end{equation}
we have that \eqref{ex06} and \eqref{ex25} hold.
Moreover, \eqref{cor21} implies \eqref{crit_exp} while
from \eqref{ex23} and \eqref{cor22} it follows \eqref{crit_expi}.
At last, from \eqref{cor21} and \eqref{cor21bis}, since
\eqref{corbis} is satisfied, we have that $(g_4)$ holds.
Then, from Examples \ref{ex0} and \ref{ex2},
Theorem \ref{ThMolt} applies to $\J_2$ in $X$.
\end{proof}


\section*{Acknowledgement}
The authors would like to thank the referee for his/her valuable comments 
which helped to improve this manuscript.



\begin{thebibliography}{99}

\bibitem{AMS} C.O. Alves, D.C. de Morais Filho and M.A. Souto,
On systems of elliptic equations involving subcritical
or critical Sobolev exponents, \emph{Nonlinear Anal.}
{\bf 42} (2000), 771-787.

\bibitem{AR} 
A. Ambrosetti and P.H. Rabinowitz, 
Dual variational methods in critical point theory and applications, 
\emph{J. Funct. Anal.} \textbf{14} (1973), 349-381.

\bibitem{AB1} D. Arcoya and L. Boccardo, Critical points
for multiple integrals of the calculus of variations, \emph{Arch.
Rational Mech. Anal.} \textbf{134} (1996), 249-274.

\bibitem{AG} G. Arioli and F. Gazzola, 
Existence and multiplicity results for quasilinear elliptic differential systems, 
\emph{Comm. Partial Differential Equations} \textbf{25} (2000),
125-153. DOI:10.1080/03605300008821510

\bibitem{BBF}
P. Bartolo, V. Benci and D. Fortunato,
{Abstract critical point theorems and applications 
to some nonlinear problems with ``strong'' resonance at infinity},
\emph{Nonlinear Anal.} \textbf{7} (1983), 981-1012.

\bibitem{BB} A. Bensoussan and L. Boccardo,
Nonlinear systems of elliptic equations with natural growth
conditions and sign conditions, 
\emph{Appl. Math. Optim.} \textbf{46} (2002), 143-166.

\bibitem{BdF} L. Boccardo and D.G. de Figueiredo,
Some remarks on a system of quasilinear elliptic equations,
\emph{NoDEA Nonlinear Differential Equations Appl.}
\textbf{9} (2002), 309-323.

\bibitem{Br} 
H. Brezis,
\emph{Functional Analysis, Sobolev Spaces and Partial Differential Equations}, 
Universitext \textbf{XIV}, Springer, New York, 2011. 

\bibitem{CMPP}
A.M. Candela, E. Medeiros, G. Palmieri and K. Perera,
Weak solutions of quasilinear elliptic systems via the cohomological index,
\emph{Topol. Methods Nonlinear Anal.} \textbf{36} (2010), 1-18.

\bibitem{CP1} A.M. Candela and G. Palmieri, {Multiple solutions
of some nonlinear variational problems}, \emph{Adv. Nonlinear Stud.} 
\textbf{6} (2006), 269-286.

\bibitem{CP2}
A.M. Candela and G. Palmieri,
{Infinitely many solutions of some nonlinear variational equations},
\emph{Calc. Var. Partial Differential Equations} \textbf{34} (2009), 495-530.

\bibitem{CP3} A.M. Candela and G. Palmieri, {Some abstract critical point theorems
and applications}. In: \emph{Dynamical Systems, Differential Equations and Applications} 
(X. Hou, X. Lu, A. Miranville, J. Su \& J. Zhu Eds), 
\emph{Discrete Contin. Dynam. Syst.} \textbf{Suppl. 2009} (2009), 133-142. 

\bibitem{CP2017}
A.M. Candela and G. Palmieri,
Multiplicity results for some nonlinear elliptic problems
with asymptotically $p$--linear terms,
\emph{Calc. Var. Partial Differential Equations} \textbf{56} (2017), Article 72 (39 pp).
DOI:10.1007/s00526-017-1170-4.

\bibitem{CPP2015}
A.M. Candela, G. Palmieri and K. Perera,
Multiple solutions for $p$--Laplacian type problems 
with asymptotically $p$--linear terms via a cohomological index theory, 
\emph{J. Differential Equations} \textbf{259} (2015), 235-263. 
DOI:10.1016/j.jde.2015.02.007,

\bibitem{CPS_CCM}
A.M. Candela, G. Palmieri and A. Salvatore,
Multiple solutions for some sym\-met\-ric supercritical problems, 
\emph{Commun. Contemp. Math.} \textbf{22} (2020), Article 1950075 (20 pp).
DOI:10.1142/S0219199719500755

\bibitem{CS2020}
A.M. Candela and A. Salvatore,
Existence of radial bounded solutions for some quasilinear elliptic
equations in $\R^N$, \emph{Nonlinear Anal.} \textbf{191} (2020), Article 111625 (26 pp).
DOI:10.1016/j.na.2019.111625.

\bibitem{Ca} A. Canino, Multiplicity of solutions for quasilinear elliptic equations,
\emph{Topol. Methods Nonlinear Anal.} \textbf{6} (1995), 357-370.

\bibitem{dF} D.G. de Figueiredo, Nonlinear elliptic systems, 
\emph{An. Acad. Brasil. Ci\^enc.} \textbf{72} (2000), 453-469.

\bibitem{DSZ} P. Dr\'abek, M.N. Stavrakakis and N.B. Zographopoulos,
Multiple nonsemitrivial solutions for quasilinear elliptic systems,
\emph{Differential Integral Equations} \textbf{16} (2003), 1519-1531.

\bibitem{FMMT}
L.F.O. Faria, O.H. Miyagaki, D. Motreanu and M. Tanaka,
Existence results for nonlinear elliptic equations with
Leray--Lions operator and dependence on the gradient,
\emph{Nonlinear Analysis} \textbf{96} (2014), 154-166.

\bibitem{GL}
R. Glowinski and A. Marrocco, 
Sur l'approximation, par éléments finis d'ordre un, et la résolution, par pénalisation-dualité,
d'une classe de problémes de Dirichlet non linéaires,
\emph{Rev. Française Automat. Informat. Recherche Opérationnelle Sér. Rouge Anal. Numér.} \textbf{9} (1975), 41-76.

\bibitem{LU} O.A. Ladyzhenskaya and N.N. Ural'tseva, \emph{Linear and Quasilinear Elliptic
Equations}, Academic Press, New York, 1968.

\bibitem{Lin} P. Lindqvist, On the equation 
${\rm div} (|\nabla u|^{p-2}\nabla u) + \lambda |u|^{p-2}u =0$, 
\emph{Proc. Amer. Math. Soc.} \textbf{109} (1990), 157-164.

\bibitem{PeSq} B. Pellacci and M. Squassina, Unbounded critical points for a class
of lower semicontinuous functionals, \emph{J. Differential Equations} \textbf{201}
(2004), 25-62.

\bibitem{PAO}
K. Perera, R.P. Agarwal and D. O'Regan,
\emph{Morse Theoretic Aspects of $p$--Laplacian Type Operators},
Math. Surveys Monogr. \textbf{161}, Amer. Math. Soc., Providence RI, 2010.

\bibitem{Ra1}
P.H. Rabinowitz,
\emph{Minimax Methods in Critical Point Theory with Applications to Differential 
Equations}, CBMS Regional Conf. Ser. in Math. \textbf{65}, Amer. Math. Soc., Providence, 1986.

\bibitem{Sq}
M. Squassina, \emph{Existence, Multiplicity, Perturbation, and
Concentration Results for a Class of
Quasi-linear Elliptic Problems}, 
Electron. J. Differ. Equ. Monogr. \textbf{7}, Texas State University–San Marcos, 
San Marcos TX, 2006.

\bibitem{VdT} J. V\'elin and F. de Th\'elin,
Existence and nonexistence of nontrivial solutions
for some nonlinear elliptic systems,
\emph{Rev. Mat. Univ. Complut. Madrid} \textbf{6} (1993), 153-194.


\end{thebibliography}
\end{document}